\documentclass[a4paper, reqno, 10pt]{amsart}
	\usepackage{hyperref}
	\usepackage[usenames,dvipsnames]{color}
	\usepackage{amsthm,amsfonts,amssymb,amsmath,amsxtra, array}
	\usepackage[utf8]{inputenc}
	\usepackage{comment}
	\usepackage{enumitem}
	\usepackage[all]{xy}
	\usepackage{xr-hyper}
	\usepackage{verbatim}
	\usepackage{tikz-cd, epic}
	\usepackage[margin=1.25in]{geometry}
	\usepackage{mathrsfs}
	\usepackage{bbm}
	\usepackage{float}
	\usepackage{bigints}
	
	\newtheorem{theorem}{Theorem}
	\newtheorem{proposition}[theorem]{Proposition}
	\newtheorem{lemma}[theorem]{Lemma}
	
	\newtheorem{corollary}[theorem]{Corollary}
	
	\newtheorem*{theorema}{Theorem A}
	\newtheorem*{theoremb}{Theorem B}
    \newtheorem*{corthma}{Corollary to Theorem A}
	\theoremstyle{definition}

	\newtheorem{remark}[theorem]{Remark}

	\newtheorem*{corollary to thm}{Corollary to Theorem 0.5}

	\numberwithin{equation}{section}
	\numberwithin{theorem}{section}
	
	\title[Integrality of locally algebraic representations of $\mathrm{GL}_{2}(D)$]{On the integrality of locally algebraic representations of $\mathrm{GL}_{2}(D)$}
	\author{Santosh Nadimpalli}
	\address{Department of Mathematics and Statistics,
		Indian Institute of Technology Kanpur, Kanpur - 208016, India}
	\email{nsantosh@iitk.ac.in}
	\author{Mihir Sheth}
	\address{Department of Mathematics, Indian Institute of Science, Bangalore - 560012, India}
	\email{mihirsheth@iisc.ac.in}
	\date{}
	
\begin{document}
	
	\begin{abstract}
	
          \noindent Emerton's theory of Jacquet modules for locally analytic representations provides necessary conditions for the existence of integral structures in locally analytic representations. These conditions are also expected to be sufficient for the integrality of generic irreducible locally algebraic representations. In this article, we prove the sufficiency of Emerton's conditions for some tamely ramified locally algebraic representations of $\mathrm{GL}_{2}(D)$ where $D$ is a $p$-adic division algebra.
	\end{abstract}
	
	\maketitle
	
	\section{Introduction}
	
	Let $p$ be a prime number, $F$ be a finite extension of
        $\mathbb{Q}_{p}$ with residue field $\mathbb{F}_{q}$
        and uniformizer $\varpi_{F}$, and let $E$ be a large enough
        finite extension of $F$. Let $G$ be the group of rational points of a connected reductive group over $F$ and $\pi=\pi_{sm}\otimes\pi_{alg}$
        be an irreducible locally algebraic representation of $G$
        over $E$. The question of the existence of a
        $G$-invariant lattice or a $G$-integral structure
        in $\pi$ is of fundamental interest to the $p$-adic Langlands program.
	
	Emerton gives necessary conditions for the existence of integral structures in terms of the \emph{exponents} of Jacquet modules of locally algebraic representations. Let $P=M_{P}N_{P}\subseteq G$ be a parabolic subgroup with the modulus character $\delta_{P}$ and let $J_{P}$ denote Emerton's locally analytic Jacquet module functor. If $\pi$ admits a $G$-integral structure, then for every parabolic $P$ of $G$ and $\chi$ such that $\mathrm{Hom}_{Z(M_{P})}(\chi,J_{P}(\pi))\neq 0$,
	\begin{equation}\label{emerton_condition}
	    \text{$(\delta_{P}^{-1}\chi)(z)$ is integral in $E$},
	\end{equation} for all $z\in Z(M_{P})^{+}$ where $Z(M_{P})^{+}$ is the contracting monoid in the center $Z(M_{P})$ of the Levi factor $M_{P}$ \cite[Lemma 4.4.2]{eme06}. The characters $\chi$ of $Z(M_{P})$ occurring in $J_{P}(\pi)$ are called exponents. It is expected that the condition \eqref{emerton_condition} is also sufficient for the existence of an integral structure in $\pi$ when $\pi_{sm}$ is \emph{generic}. For $G=\mathrm{GL}_{n}(F)$, this is equivalent to Breuil-Schneider conjecture (see Hu \cite{hu09}). Note that for $P=G$, the condition \eqref{emerton_condition} says that the central character of $\pi$ is integral. When $\pi$ has integral central character and $\pi_{sm}$ is essentially square-integrable, the Jacquet modules $J_{P}(\pi)$ for proper parabolic $P$ always satisfy \eqref{emerton_condition}. In this situation, Sorensen showed using global methods that the integrality of the central character is sufficient for $\pi$ to have an integral structure \cite{sor13}. On the other hand, when $\pi_{sm}$ is a principal series representation, the Jacquet modules are no longer simple and one requires further conditions on $\pi$ whose sufficiency is not easy to prove. There are only partial results available, even for $\mathrm{GL}_{2}(F)$, when $\pi_{sm}$ is an unramified principal series and the weights of $\pi_{alg}$ are small \cite{di13, ass21} or when $\pi=\pi_{sm}$ is a tamely ramified principal series \cite{vig08}. For general split reductive groups and $F=\mathbb{Q}_{p}$, Gro\ss e-Kl\"{o}nne has constructed integral structures in unramified smooth principal series representations under some technical hypothesis \cite{gk14}.   
	
	In this article, we consider the non-quasi-split group $G=\mathrm{GL}_{2}(D)$ where $D$ is a central $F$-division algebra of dimension $d^{2}$ and show that \eqref{emerton_condition} is sufficient for the existence of integral structures in some tamely ramified irreducible locally algebraic representations of $G$. Let us spell out \eqref{emerton_condition} for representations of $G=\mathrm{GL}_{2}(D)$ admitting integral structures. Let $B=TN$ be the minimal parabolic subgroup of $G$ consisting of upper triangular matrices. One has $\delta_{B}(z)=q^{-d^{2}}$ for $z=\left(\begin{smallmatrix} \varpi_{F} & 0 \\ 0 & 1\end{smallmatrix}\right)\in Z(T)^{+}$. Denote by $\pi(\underline{\lambda})$ the irreducible algebraic representation of $\mathrm{GL}_{2d}(\overline{F})$ with highest weight $\underline{\lambda}=(\lambda_{1},\ldots,\lambda_{2d})$ and by $\chi(\underline{\lambda})$ the character $(t_{1},\ldots,t_{2d})\mapsto t_{1}^{\lambda_{1}}\ldots t_{2d}^{\lambda_{2d}}$ of its diagonal torus.
	
	If $\pi=\mathrm{Ind}_{B}^{G}(\tau_{1}\otimes\tau_{2})\otimes\pi(\underline{\lambda})$ is a locally algebraic principal series representation, then \[J_{B}(\pi)\cong(\mathrm{Ind}_{B}^{G}(\tau_{1}\otimes\tau_{2}))_{N}\otimes\pi(\underline{\lambda})^{N}\] and \[(\mathrm{Ind}_{B}^{G}(\tau_{1}\otimes\tau_{2}))_{N}^{ss}\cong(\tau_{1}\otimes\tau_{2})\oplus((\tau_{2}\otimes\tau_{1})\otimes\delta_{B}).\] Denoting the central characters of representations $?$ by $\omega_{?}$, the exponents $\chi$ of $J_{B}(\pi)$ are \begin{equation*}
	\text{$(\omega_{\tau_{1}}\otimes\omega_{\tau_{2}})\chi(\underline{\lambda})$ and $(\omega_{\tau_{2}}\otimes\omega_{\tau_{1}})\delta_{B}\chi(\underline{\lambda})$.}\end{equation*}
	Thus, if $\pi$ has an integral structure, \eqref{emerton_condition} says that \begin{align}\label{emerton_condition_ps}
	    &\text{$(\omega_{\tau_{1}}\omega_{\tau_{2}})(\varpi_{F})\varpi_{F}^{\sum_{i=1}^{2d}\lambda_{i}}$is an integral unit (the integrality of $\omega_{\pi}$), and} \nonumber \\& \hspace{8mm}\text{$q^{d^{2}}\omega_{\tau_{1}}(\varpi_{F})\varpi_{F}^{\sum_{i=1}^{d}\lambda_{i}}$ and $\omega_{\tau_{2}}(\varpi_{F})\varpi_{F}^{\sum_{i=1}^{d}\lambda_{i}}$ are integral in $E$}.
	\end{align}
	Our first main result shows that the conditions \eqref{emerton_condition_ps} are sufficient for the existence of an integral structure in $\pi$ if the smooth principal series  $\mathrm{Ind}_{B}^{G}(\tau_{1}\otimes\tau_{2})$ is tamely ramified and $\pi(\underline{\lambda})$ is trivial, i.e., $\underline{\lambda}=\underline{0}$:
	\begin{theorema}[Theorem \ref{main_thm_for_prin_ser}]\label{theorema}
		The integrality conditions \eqref{emerton_condition_ps} of Emerton are sufficient for the existence of an integral structure in a smooth tamely ramified principal series representation $\mathrm{Ind}_{B}^{G}(\tau_{1}\otimes\tau_{2})$ of $G$. 
	\end{theorema}
	
	We remark that the principal series $\mathrm{Ind}_{B}^{G}(\tau_{1}\otimes\tau_{2})$ is not required to be irreducible in Theorem A. However, the conditions \eqref{emerton_condition_ps} are no longer sufficient for a reducible principal series tensored with a non-trivial algebraic representation (see erratum to \cite{vig08}).
	
	A principal series of the form $\mathrm{Ind}_{B}^{G}((\tau\otimes\tau)\otimes\delta_{B}^{\frac{d-a}{2d}})$ is reducible with unique irreducible quotient $\mathrm{St}(\tau)$ and unique irreducible submodule $\mathrm{Sp}(\tau)$. Here $a$ is the length of the segment that determines the Jacquet-Langlands lift of the irreducible $D^{\times}$-representation $\tau$. If $\pi=\mathrm{St}(\tau)\otimes\pi(\underline{\lambda})$ (resp. $\mathrm{Sp}(\tau)\otimes\pi(\underline{\lambda})$), then \[J_{B}(\pi)\cong((\tau\otimes\tau)\otimes\delta_{B}^{\frac{d+a}{2d}})\otimes\pi(\underline{\lambda})^{N}\hspace{3mm}(\text{resp.}\hspace{2mm}((\tau\otimes\tau)\otimes\delta_{B}^{\frac{d-a}{2d}})\otimes\pi(\underline{\lambda})^{N}).\] The exponent $\chi$ in $J_{B}(\pi)$ is $(\omega_{\tau}\otimes\omega_{\tau})\delta_{B}^{\frac{d+a}{2d}}\chi(\underline{\lambda})$ (resp. $(\omega_{\tau}\otimes\omega_{\tau})\delta_{B}^{\frac{d-a}{2d}}\chi(\underline{\lambda})$). Hence, if $\pi$ has an integral structure, then \eqref{emerton_condition} says that \begin{align}\label{emerton_condition_st_sp}
	    &\text{$\omega_{\tau}^{2}(\varpi_{F})\varpi_{F}^{\sum_{i=1}^{2d}\lambda_{i}}$ is an integral unit and $q^{\frac{d(d-a)}{2}}\omega_{\tau}(\varpi_{F})\varpi_{F}^{\sum_{i=1}^{d}\lambda_{i}}$}\nonumber \\& \text{(resp. $q^{\frac{d(d+a)}{2}}\omega_{\tau}(\varpi_{F})\varpi_{F}^{\sum_{i=1}^{d}\lambda_{i}}$) is integral in $E$.}
	\end{align}
	The first part of \eqref{emerton_condition_st_sp}, i.e., the integrality of $\omega_{\pi}$ implies that $\mathrm{val}_{E}(\omega_{\tau}(\varpi_{F}))=\frac{-1}{2}\sum_{i=1}^{2d}\lambda_{i}$. Thus $\mathrm{val}_{E}(\omega_{\tau}(\varpi_{F}))+\sum_{i=1}^{d}\lambda_{i}=\frac{1}{2}(\sum_{i=1}^{d}\lambda_{i}-\sum_{i=d+1}^{2d}\lambda_{i})\geq 0$. As $d-a\geq 0$, we see that the second part of \eqref{emerton_condition_st_sp} is redundant as it is implied by the first part. Hence, in this case, the integrality of $\omega_{\pi}$ is conjecturally sufficient for the existence of an integral structure in $\pi$. The sufficiency follows from Theorem \ref{main_thm_for_prin_ser} when  $\mathrm{St}(\tau)$ (resp. $\mathrm{Sp}(\tau)$) is tamely ramified and $\pi(\underline{\lambda})$ is trivial:
	
	\begin{corthma}[Theorem \ref{main_thm_for_st_and_sp}]
 A tamely ramified Steinberg representation $\mathrm{St}(\tau)$ with integral central character admits an integral structure. A tamely ramified Speh representation $\mathrm{Sp}(\tau)$ with integral central character admits an integral structure.
	\end{corthma} 

    Our second main result shows the sufficiency of the integrality of $\omega_{\pi}$ for some examples of locally algebraic representations $\pi=\mathrm{St}(\tau)\otimes\pi(\underline{\lambda})$ (resp. $\mathrm{Sp}(\tau)\otimes\pi(\underline{\lambda})$) with non-trivial $\pi(\underline{\lambda}
	)$:
	
	\begin{theoremb}[Theorem \ref{main_thm__for_la_st_and_sp} and Theorem \ref{la_sp_int}]
		\begin{enumerate}
			\item[]
			\item Let $\tau$ be a smooth absolutely irreducible tamely ramified representation of $D^{\times}$ of dimension $\leq 2$. Then the locally algebraic Steinberg representation $\pi=\mathrm{St}(\tau)\otimes\pi(\underline{\lambda})$ of $G$ with integral central character admits an integral structure.
			\item Let $D$ be the quaternionic division algebra and $\tau$ be a smooth absolutely irreducible tamely ramified representation of $D^{\times}$ of dimension $2$. Then the locally algebraic Speh representation $\mathrm{Sp}(\tau)\otimes(\mathrm{Sym}^{1}E^{4}\otimes\mathrm{det}^{-\frac{1}{4}})$ of $G$ admits an integral structure.
		\end{enumerate}
	\end{theoremb}
	
	Though part (1) of Theorem B is implied by \cite[Theorem B]{sor13}, our proofs of both Theorem A and Theorem B are purely local, unlike the ones in \textit{loc. cit.} We follow methods of Vign\'{e}ras \cite{vig08} and Hu \cite{hu20} based on the theory of coefficient systems (or \emph{diagrams}) on the Bruhat-Tits tree of $G$. The main idea is to use the realization of a locally algebraic representation $\pi$ as the $0$-th homology group of its fixed-point system. The question of finding an integral structure in $\pi$ then amounts to the question of finding a system of lattices in the corresponding fixed-point system of finite-dimensional vector spaces. The analysis gets more involved when $\pi_{sm}$ admits invariants under smaller compact open subgroups and when $\pi_{alg}$ is non-trivial. For example, the proof of integrality in part (2) of Theorem B involves some tedious computations with lattices and the integral structure is obtained at the second iteration of the ``zig-zag" sequence of lattices. For part (1) of Theorem B, a generalization of Hu's argument allows us to treat $\pi$ with \emph{any} non-trivial $\pi_{alg}$ (without explicitly working with $\pi_{alg}$) under a strong assumption on $\mathrm{dim}_{E}(\tau)$ which is necessary (cf. Lemma \ref{I(1)-invariants of St and Sp} (ii)).
	
	The integrality of the locally algebraic Speh representation in part (2) of Theorem B is an interesting phenomenon and is related to the \emph{genericity} of the Speh representation. Note that a Speh representation $\mathrm{Sp}(\tau)$ is finite-dimensional (and thus one-dimensional) if and only if $\tau$ is a character. An infinite-dimensional $\mathrm{Sp}(\tau)$, as the one in part (2) of Theorem B, admits a \emph{non-degenrate Whittaker model} and thus is generic (cf. \cite[Corollary 8.3 and Remark 8.4 (2)]{guy2023representations} or \cite[Theorem 3.1]{pr00}). The tensor product of a one-dimensional Speh representation and a non-trivial irreducible algebraic representation is never integral. However, it is expected that locally algebraic generic Speh representations with integral central characters are integral. The genericity of the smooth part is also important in the formulation of the Breuil-Schneider conjecture which predicts the existence of integral structures in locally algebraic representations
of $\mathrm{GL}_{n}(F)$ \cite[p. 16]{bs07}. As any \emph{infinite-dimensional} smooth irreducible representation of $\mathrm{GL}_{2}(D)$ is generic, our investigations suggest that Emerton's integrality conditions \eqref{emerton_condition} are sufficient for the integrality of any \emph{infinite-dimensional} irreducible locally algebraic representation of $\mathrm{GL}_{2}(D)$. We hope to address this uniformly in future work.
	%
	\\
	
	\noindent {\it Organization}: In \S \ref{section2}, we discuss Vign\'eras' integrality criterion for the representations of $\mathrm{GL}_{2}(D)$. In \S \ref{section3}, we use this criterion to show that Emerton's integrality conditions are sufficient for the integrality of smooth tamely ramified principal series representations. The \S \ref{section4} talks about the integrality of locally algebraic representations of $\mathrm{GL}_{2}(D)$ whose smooth part is either a Steinberg representation or a Speh representation. We show the sufficiency of Emerton's conditions for some locally algebraic Steinberg representations. Then, finally, we illustrate the connection between the integrality of locally algebraic representations and the \emph{genericity} of their smooth part with an example of an infinite-dimensional integral locally algebraic Speh representation. 
	\\
	\\
	\noindent {\it Acknowledgments.} We thank the anonymous referee for several useful comments on an earlier version of this paper. During this work, the first author was supported by the DST-INSPIRE Research Grant of the Govt of India and the second author was supported by the Raman Postdoctoral Fellowship of the Indian Institute of Science, Bangalore. Both authors acknowledge the support that they received. 
	\\
	\\
	{\bf Notation and convention}: 
	
	Let $F$ be a non-archimedean local field of characteristic $0$ with residue field $\mathbb{F}_{q}$ of characteristic $p$. Let $D$ be the central $F$-division algebra of index $d$. Let $\mathcal{O}_{F}\subseteq F$ and $\mathcal{O}_{D}\subseteq D$ denote the respective valuation rings. Fix uniformizers $\varpi_{F}\in\mathcal{O}_{F}$ and $\varpi_{D}\in\mathcal{O}_{D}$. For a divisor $d'$ of $d$, let $F_{d'}$ denote the unramified extension of $F$ of degree $d'$ viewed as a subfield of $D$. Let $|\cdot|_{F}$ and $|\cdot|$ denote the normalized non-archimedean absolute values on $F$ and $D$ respectively such that $|\varpi_{F}|_{F}=q^{-1}$ and $|\varpi_{D}|=q^{-d}$. We have $|\cdot|:=|\cdot|^{d}_{F}\circ\mathrm{Nrd}$ where $\mathrm{Nrd}:D\rightarrow F$ is the reduced norm. Note that $|\varpi_{F}|=q^{-d^{2}}$.
	
	Let $G$ be the group $\mathrm{GL}_{2}(D)$ of units in the matrix algebra $\mathrm{M}_{2}(D)$, $K=\mathrm{GL}_{2}(\mathcal{O}_{D})$ and $I\subseteq K$ denote the standard Iwahori subgroup. We view $D^{\times}$ as a subgroup of $G$ embedded diagonally in it. Let $\mathfrak{K}_{0}, \mathfrak{K}_{1}\subseteq G$ be the subgroups stabilizing respectively the standard vertex and the standard edge of the Bruhat-Tits tree of $G$. Note that $\mathfrak{K}_{0}=K\varpi_{D}^{\mathbb{Z}}$, and $\mathfrak{K}_{1}$ is generated by $I$ and the matrix $t=\left(\begin{smallmatrix} 0 & 1 \\ \varpi_{D} & 0\end{smallmatrix}\right)$. The groups $K$ and $I$ admit filtrations by pro-$p$-subgroups $K(n)$ and $I(n)$ respectively for $n\geq 1$ where $K(n)=\left(\begin{smallmatrix} 1+\varpi_{D}^{n}\mathcal{O}_{D} & \varpi_{D}^{n}\mathcal{O}_{D} \\ \varpi_{D}^{n}\mathcal{O}_{D} & 1+\varpi_{D}^{n}\mathcal{O}_{D}\end{smallmatrix}\right)$ and $I(n)=\left(\begin{smallmatrix} 1+\varpi_{D}^{n}\mathcal{O}_{D} & \varpi_{D}^{n-1}\mathcal{O}_{D} \\ \varpi_{D}^{n}\mathcal{O}_{D} & 1+\varpi_{D}^{n}\mathcal{O}_{D}\end{smallmatrix}\right)$. The subgroup $I(1)\subseteq I$ is the standard pro-$p$-Iwahori subgroup. The groups $\mathfrak{K}_{0}$ and $\mathfrak{K}_{1}$ are the normalizers of $K(n)$ and $I(n)$ in $G$ respectively for all $n\geq 1$. Let $B\subseteq G$ be the subgroup of upper triangular matrices (the standard minimal parabolic subgroup), $N\subseteq B$ be the subgroup of upper triangular unipotent matrices (the unipotent radical of $B$), and $T\subseteq B$ be the subgroup of diagonal matrices (the Levi quotient of $B$). The modulus character $\delta_{B}$ of $T$ is $|\cdot|\otimes|\cdot|^{-1}$. We let $Z$ denote the center of $G$ which is isomorphic to $F^{\times}$.
	
	We fix a large enough finite extension $E$ of $F$. The field $E$ depends on the representation at hand and we will explain how large it should be at various places in the article when required. Its valuation ring is denoted by $\mathcal{O}$ and the residue field is denoted by $k=\mathcal{O}/\varpi\mathcal{O}$ where $\varpi\in\mathcal{O}$ is a uniformizer. The rings $R= E, \mathcal{O}, k$ will serve as the coefficient rings for representations of $G$. The representations will be either denoted by $(\pi, V)$ or just by $\pi$ or $V$ depending on the situation. Let $H\subseteq G$ be a subgroup. We write $RH$ for the group algebra of $H$ over $R$ and use the phrases ``$RH$-modules" and ``representations of $H$ over $R$" interchangeably. If $V$ is an $RH$-module, then, for a subset $S\subseteq V$, we denote by $H\cdot S$ the $RH$-submodule of $V$ generated by $S$. 
	
	We fix an isomorphism between $\mathbb{C}$ and the algebraic closure $\overline{E}$ of $E$. For a smooth representation $\pi$ over $E$, we write $\pi_{\mathbb{C}}=\pi\otimes_{E}\mathbb{C}$ for its scalar extension via the embedding $E\hookrightarrow\mathbb{C}$ induced by the fixed isomorphism. We also call $\pi$ an $E$-model of $\pi_{\mathbb{C}}$. By \cite[Section 3.13]{ceg+}, all the results about $\pi_{\mathbb{C}}$ are valid for $\pi$ (over a large enough $E$).
	If a representation $\pi$ admits a central character, it is denoted by $\omega_{\pi}$.
	
	A smooth representation of $G$ (resp. of $D^{\times}$) will be called tamely ramified if it has a non-zero vector fixed by the subgroup $K(1)$ (resp. by $D(1)=1+\varpi_{D}\mathcal{O}_{D}$). For a divisor $d'$ of $d$, let $D_{d'}$ be the centralizer of $F_{d'}$ in $D$ which is a central $F_{d'}$-division algebra of index $d/d'$. Let $\theta:F_{d'}^{\times}\rightarrow E^{\times}$ be a character which is trivial on the subgroup of integral units congruent to $1$ modulo the maximal ideal and all whose Galois conjugates are distinct. Here one requires $E$ to contain all $d$ roots of unity. Composing it with the reduced norm $\mathrm{Nrd}:D_{d'}^{\times}\rightarrow F_{d'}^{\times}$ and extending it to $D(1)D^{\times}_{d'}$ by declaring it to be trivial on $D(1)$, we get a character, denoted by the same letter, $\theta:D(1)D^{\times}_{d'}\rightarrow E^{\times}$. The representation $\mathrm{Ind}_{D(1)D^{\times}_{d'}}^{D^{\times}}\theta$ is absolutely irreducible and tamely ramified. In fact, all smooth tamely ramified absolutely irreducible representations of $D^{\times}$ over $E$ are obtained in this fashion \cite{sz05}. 
	
	\section{Coefficient systems and Vign\'{e}ras' integrality criterion}\label{section2}
	
	We begin by recalling some definitions. Let $H\subseteq G$ be an open subgroup. A locally algebraic representation of $H$ over $E$ is a representation of the form $\pi=\pi_{sm}\otimes\pi_{alg}$, where $\pi_{sm}$ is a smooth representation of $H$ over $E$ and $\pi_{alg}$ is the restriction to $H$ of a finite-dimensional rational representation of $G$ over $E$. If $\pi_{sm}$ has a name ``X", then $\pi$ will be called by the name ``locally algebraic X". The representation $\pi$ is irreducible if and only if $\pi_{sm}$ and $\pi_{alg}$ are irreducible \cite[Appendix, Theorem 1]{st01}. An $H$-integral structure is an $H$-stable free $\mathcal{O}$-submodule $\pi^{\circ}\subseteq\pi$ which spans $\pi$ over $E$. The integral structure is also called an $H$-lattice since it is a lattice stable under the action of $H$. If an $H$-integral structure exists, we say that $\pi$ is $H$-integral, or just integral if the group is clear from the context. We are interested in the integrality of irreducible locally algebraic representations of $G$. 
	
	A diagram \[\mathcal{D}_{1}\xrightarrow{r}\mathcal{D}_{0}\] is a datum consisting of continuous (smooth when $R=k$) $R\mathfrak{K}_{i}$-modules $\mathcal{D}_{i}$ and a map $r$ equivariant for the action of $\mathfrak{K}_{0}\cap\mathfrak{K}_{1}=I\varpi_{D}^{\mathbb{Z}}$. Such a diagram gives rise to a $G$-equivariant coefficient system on the Bruhat-Tits tree of $G$. Conversely, the restriction of a $G$-equivariant coefficient system to the subtree consisting of the standard edge and the standard vertex is a diagram. Associated to a diagram $\mathcal{D}$ (or to a coefficient system), one has oriented chain homology groups $H_{i}(\mathcal{D})$, $i=0,1$, which are continuous $RG$-modules defined by the following exact sequence \[0\xrightarrow{\hspace{3mm}} H_{1}(\mathcal{D})\xrightarrow{\hspace{3mm}}\text{c-Ind}_{\mathfrak{K}_{1}}^{G}\mathcal{D}_{1}^{\varepsilon}\xrightarrow{\hspace{1mm}\partial\hspace{1mm}}\text{c-Ind}_{\mathfrak{K}_{0}}^{G}\mathcal{D}_{0}\xrightarrow{\hspace{3mm}} H_{0}(\mathcal{D})\xrightarrow{\hspace{3mm}} 0.\] Here, the $\mathfrak{K}_{1}$-action on $\mathcal{D}_{1}^{\varepsilon}$ is twisted by the non-trivial character $\varepsilon$ on $\mathfrak{K}_{1}/(\mathfrak{K}_{0}\cap\mathfrak{K}_{1})\cong\mathbb{Z}/2\mathbb{Z}$, and the $G$-equivariant boundary map $\partial$ is given by $\partial([1,v])=[1,r(v)]-[t,r(t^{-1}v)]$ for $v\in\mathcal{D}_{1}^{\varepsilon}$, where $[g,v]\in\text{c-Ind}_{\mathfrak{K}_{1}}^{G}\mathcal{D}_{1}^{\varepsilon}$ (resp. $[g,v]\in\text{c-Ind}_{\mathfrak{K}_{0}}^{G}\mathcal{D}_{0}$) denotes the unique function supported on $\mathfrak{K}_{1}g^{-1}$ (resp. $\mathfrak{K}_{0}g^{-1}$) taking value $v\in\mathcal{D}_{1}^{\varepsilon}$ (resp. $v\in\mathcal{D}_{0}$) on $g^{-1}$.
	
	Let $\pi=\pi_{sm}\otimes\pi_{alg}$ be a locally algebraic representation of $G$ over $E$. Assume that $\pi_{sm}$ is generated by its subspace $\pi_{sm}^{I(n)}$ of $I(n)$-invariants for some positive integer $n$ and $\mathrm{dim}_{E}(\pi_{sm}^{I(n)})<\infty$. Let $V_{1}=\pi_{sm}^{I(n)}\otimes\pi_{alg}$ and $V_{0}=\pi_{sm}^{K(n)}\otimes\pi_{alg}$ and consider the diagram 
	\[\mathcal{D}(\pi)=V_{1}\hookrightarrow V_{0}\] of $E\mathfrak{K}_{i}$-modules $V_{i}$. It follows from \cite[Theorem II.3.1]{ss97} and \cite[Proposition 0.4]{vig08} (with exactly the same proof for $\mathrm{GL}_{2}(D)$) that the representation $H_{0}(\mathcal{D}(\pi))$ of $G$ is isomorphic to $\pi$. 
	
	\begin{theorem}[Vign\'{e}ras]\label{vigneras_criterion}
	$\pi$ is integral if and only if $V_{0}$ contains a $\mathfrak{K}_{0}$-lattice $M_{0}$ such that $M_{1}=M_{0}\cap V_{1}$ is a $\mathfrak{K}_{1}$-lattice of $V_{1}$, i.e. $M_{1}$ is stable under the action of $t$. In this situation, $H_{0}({M_{1}\hookrightarrow M_{0}})$ is a $G$-integral structure of $\pi$.
	\end{theorem}
	
	\begin{proof}
	See \cite[Corollary 0.2 and Proposition 0.4]{vig08}.
	\end{proof}
	Suppose $V_{1}$ contains a $\mathfrak{K}_{1}$-lattice $L_{1}$. Starting from $L_{1}$, define inductively an increasing ``zig-zag" sequence of $\mathfrak{K}_{1}$-lattices of $V_{1}$ as follows: 
	\begin{align*}
	    & L_{1}^{(0)}:=L_{1},\\ & L_{1}^{(i+1)}:=\left(\left(\mathfrak{K}_{0}\cdot L_{1}^{(i)}\right)\cap V_{1}\right)+t\left(\left(\mathfrak{K}_{0}\cdot L_{1}^{(i)}\right)\cap V_{1}\right) \hspace{2mm} \text{for $i\geq 0$.}
	\end{align*}
	
	\begin{corollary}\label{zigzag_criterion}
	$\pi$ is integral if and only if $V_{1}$ contains a $\mathfrak{K}_{1}$-lattice $L_{1}$ such that the increasing sequence $(L_{1}^{(i)})_{i}$ of $\mathfrak{K}_{1}$-lattices of $V_{1}$ becomes stationary.
	\end{corollary}
	\begin{proof}
	If $\pi$ is integral, then the $\mathfrak{K}_{1}$-lattice $M_{1}$ as in Theorem \ref{vigneras_criterion} satisfies $M_{1}^{(0)}=M_{1}^{(1)}$. Conversely, if $V_{1}$ contains a $\mathfrak{K}_{1}$-lattice $L_{1}$ such that $L_{1}^{(i_{0})}=L_{1}^{(i_{0}+1)}$ for some positive integer $i_{0}$, then $M_{0}=\mathfrak{K}_{0}\cdot L_{1}^{(i_{0})}$ is a $\mathfrak{K}_{0}$-lattice of $V_{0}$ such that $M_{1}=M_{0}\cap V_{1}=L_{1}^{(i_{0})}$.  
	\end{proof}
	
	\begin{remark}\label{zigzag_rmk}
	If $\pi$ is integral, then for any $\mathfrak{K}_{1}$-lattice $L_{1}$ of $V_{1}$, the sequence $(L_{1}^{(i)})_{i}$ of $\mathfrak{K}_{1}$-lattices becomes stationary. Indeed, we know that there is a $\mathfrak{K}_{1}$-lattice $M_{1}$ of $V_{1}$ such that the the sequence $(M_{1}^{(i)})_{i}$ becomes stationary. We may assume $L_{1}\subseteq M_{1}$. Since $[M_{1}:L_{1}]$ is finite, it is clear that the increasing sequence $(L_{1}^{(i)})_{i}$ also becomes stationary.
	\end{remark}
	
	We record one lemma for later use:
	\begin{lemma}\label{surjectivity_on_h0}
		Let $L_{1}$ be a $\mathfrak{K}_{1}$-lattice in $V_{1}$ and $L_{0}^{(i)}=\mathfrak{K}_{0}\cdot L_{1}^{(i)}$ for all $i\geq 0$. Then the natural map $H_{0}(L_{1}^{(i)}\hookrightarrow L_{0}^{(i)})\rightarrow H_{0}(L_{1}^{(i+1)}\hookrightarrow L_{0}^{(i+1)})$ is surjective for all $i\geq 0$.
	\end{lemma}	
	\begin{proof}
		By definition, $H_{0}(L_{1}^{(i+1)}\hookrightarrow L_{0}^{(i+1)})=\left(\text{c-Ind}_{\mathfrak{K}_{0}}^{G}(\mathfrak{K}_{0}\cdot L_{1}^{(i+1)})\right)/\text{Im}(\partial)$. We need to show that a homology class $[g,v]+\text{Im}(\partial)$ in $H_{0}(L_{1}^{(i+1)}\hookrightarrow L_{0}^{(i+1)})$ has a preimage in $H_{0}(L_{1}^{(i)}\hookrightarrow L_{0}^{(i)})$. As the natural map between homologies is $G$-equivariant, it is enough to consider a class $[1,v]+\text{Im}(\partial)$ in $H_{0}(L_{1}^{(i+1)}\hookrightarrow L_{0}^{(i+1)})$ with $v\in L_{1}^{(i+1)}$. An element $v\in L_{1}^{(i+1)}$ is of the form $v'+tv''$ with $v',v''\in(\mathfrak{K}_{0}\cdot L_{1}^{(i)})\cap V_{1}$. The homology class $[1,v']+\text{Im}(\partial)\in H_{0}(L_{1}^{(i+1)}\hookrightarrow L_{0}^{(i+1)})$ has a preimage $[1,v']+\text{Im}(\partial)\in H_{0}(L_{1}^{(i)}\hookrightarrow L_{0}^{(i)})$. Consider the homology class $[1,tv'']+\text{Im}(\partial)\in H_{0}(L_{1}^{(i+1)}\hookrightarrow L_{0}^{(i+1)})$. Note that $\partial([\varpi_{D}t^{-1},v''])=[\varpi_{D}t^{-1},v'']-[1,tv'']$. Thus $[1,tv'']+\text{Im}(\partial)=[\varpi_{D}t^{-1},v'']+\text{Im}(\partial)$ in $H_{0}(L_{1}^{(i+1)}\hookrightarrow L_{0}^{(i+1)})$ and the class $[\varpi_{D}t^{-1},v'']+\text{Im}(\partial)$ has a preimage $[\varpi_{D}t^{-1},v'']+\text{Im}(\partial)$ in $H_{0}(L_{1}^{(i)}\hookrightarrow L_{0}^{(i)})$ because $v''\in L_{0}^{(i)}$.
	\end{proof}
	
	The integrality criterion in Corollary \ref{zigzag_criterion} will be used in the following sections to show that Emerton's conditions are sufficient for the existence of integral structures in $\pi$ for which $\pi_{sm}$ is tamely ramified. 
	
	\section{Integrality of smooth principal series}\label{section3}
	
	Let $(\tau_1, W_1)$ and $(\tau_2, W_2)$ be two smooth absolutely irreducible tamely ramified representations of $D^\times$ over $E$. In this section, $\pi=\pi_{sm}={\rm Ind}_B^G(\tau_1\otimes \tau_2)$, a smooth principal series representation. Note that $\pi^{I(1)}\neq 0$. In fact, $\pi$ is generated by $\pi^{I(1)}$ as a $G$-representation. In order to describe the spaces $V_{0}=\pi^{K(1)}$ and $V_{1}=\pi^{I(1)}$, we define some explicit elements of the principal series $\pi$. 
	
	Let $s=\left(\begin{smallmatrix}0 & 1\\1 & 0\end{smallmatrix}\right)$ and $u_\lambda=\left(\begin{smallmatrix}1&[\lambda]\\0&1\end{smallmatrix}\right)$ where $[\lambda]\in\mathcal{O}_{D}$ is the Teichm\"{u}ller lift of $\lambda\in \mathbb{F}_{q^{d}}$. For $h\in G$ and $v\in W_1\otimes W_2$, we denote by $f^{h}_{v}$ the unique function in $\pi^{I(1)}$ supported on $BhI(1)$ such that $f^{h}_{v}(h)=v$. Similarly, we denote by $g^{h}_{v}$ the unique function in $\pi^{K(1)}$ supported on $BhK(1)$ such that $g^{h}_{v}(h)=v$. Note that $f^{1}_{v}=g^{1}_{v}$ because $BI(1)=BK(1)$, and $f^{s}_{v}=\sum_{\lambda\in\mathbb{F}_{q^{d}}}g^{su_{\lambda}}_{v}$ because $BsI(1)=\bigsqcup_{\lambda\in\mathbb{F}_{q^{d}}} Bsu_{\lambda}K(1)$. If $M\subseteq W_{1}\otimes W_{2}$ is an $\mathcal{O}$-submodule, it is convenient to write $f^{h}_{M}$ for the set $\lbrace f^{h}_{v}:v\in M\rbrace$ and similarly $g^{h}_{M}$ for the set $\lbrace g^{h}_{v}:v\in M\rbrace$. These sets are $\mathcal{O}$-modules isomorphic to $M$ for point-wise addition and scalar multiplication. As \[G=BI(1)\sqcup BsI(1)=BK(1)\sqcup\bigsqcup_{\lambda\in\mathbb{F}_{q^{d}}}Bsu_{\lambda}K(1),\] we have \begin{equation*}
	    V_{1}=f^{1}_{W_{1}\otimes W_{2}}\oplus f^{s}_{W_{1}\otimes W_{2}}\hspace{2mm}\text{and}\hspace{2mm}
	    V_{0}=g^{1}_{W_{1}\otimes W_{2}}\oplus\bigoplus_{\lambda\in\mathbb{F}_{q^{d}}} g^{su_{\lambda}}_{W_{1}\otimes W_{2}}\hspace{2mm}.
	\end{equation*}
	It is easy to check $tf^{1}_{v}=f^{s}_{(\tau_{1}(\varpi_{D})\otimes\mathrm{Id})(v)}$ and $tf^{s}_{v}=f^{1}_{(\mathrm{Id}\otimes\tau_{2}(\varpi_{D}))(v)}$. By letting $K(1)$ act trivially on $W_{1}\otimes W_{2}$, one can extend the action of $B\cap\mathfrak{K}_{0}$ on $W_{1}\otimes W_{2}$ to $I\varpi_{D}^{\mathbb{Z}}=(B\cap\mathfrak{K}_{0})K(1)$. Then, as $E\mathfrak{K}_{0}$-modules, $V_{0}\cong\mathrm{Ind}^{\mathfrak{K}_{0}}_{I\varpi_{D}^{\mathbb{Z}}}(\tau_{1}\otimes\tau_{2})$.
	
	Let $T_{0}=T\cap\mathfrak{K}_{0}$. The central character $\omega_{\pi}$ of $\pi$ equals $\omega_{\tau_{1}}\omega_{\tau_{2}}$. When $\omega_{\pi}$ is integral, there exists a $T_{0}$-lattice $\mathcal{L}\subseteq W_1\otimes W_2$. The main result of this section is the following:
	
	\begin{theorem}\label{main_thm_for_prin_ser}
	The tamely ramified principal series representation $\pi={\rm Ind}_B^G(\tau_1\otimes \tau_2)$ with integral central character is integral if and only if $q^{d^{2}}\omega_{\tau_{1}}(\varpi_{F})\in\mathcal{O}$ and $\omega_{\tau_{2}}(\varpi_{F})\in\mathcal{O}$.
	\end{theorem}
	\begin{proof}

	$(\implies)$ Though the necessity is known due to Emerton, we provide a proof to set up the notation for the next part. Let \[L_{0}:=\mathrm{Ind}^{\mathfrak{K}_{0}}_{I\varpi_{D}^{\mathbb{Z}}}\mathcal{L}=g^{1}_{\mathcal{L}}\oplus\bigoplus_{\lambda\in \mathbb{F}_{q^{d}}}g^{su_{\lambda}}_{\mathcal{L}}\] be a $\mathfrak{K}_{0}$-lattice of $V_{0}$. Then, \[L_{0}\cap V_{1}=L_{0}^{I(1)}=f^{1}_{\mathcal{L}}\oplus f^{s}_{\mathcal{L}}.\] Let \[L_{1}=L_{1}^{(0)}:=L_{0}\cap V_{1}+t(L_{0}\cap V_{1}).\] As $\mathcal{L}$ is stable under the diagonal action of $\varpi_{D}$, $L_{1}$ is stable under the action of $\mathfrak{K}_{1}$ and so it is a $\mathfrak{K}_{1}$-lattice of $V_{1}$. One computes that \[L_{1}=f^{1}_{\mathcal{L}+(\mathrm{Id}\otimes\tau_{2}(\varpi_{D}))\mathcal{L}}\oplus f^{s}_{\mathcal{L}+(\tau_{1}(\varpi_{D})\otimes\mathrm{Id})\mathcal{L}}.\]
	Thus, \begin{equation}\label{eq1}
	    \mathfrak{K}_{0}\cdot L_{1}=\mathfrak{K}_{0}\cdot f^{1}_{\mathcal{L}+(\mathrm{Id}\otimes\tau_{2}(\varpi_{D}))\mathcal{L}}+\mathfrak{K}_{0}\cdot f^{s}_{\mathcal{L}+(\tau_{1}(\varpi_{D})\otimes\mathrm{Id})\mathcal{L}}.
	\end{equation}
	Since $u_{-\lambda}s\cdot f^{1}_{v}=u_{-\lambda}s\cdot g^{1}_{v}=g^{su_{\lambda}}_{v}$, the first summand $\mathfrak{K}_{0}\cdot f^{1}_{\mathcal{L}+(\mathrm{Id}\otimes\tau_{2}(\varpi_{D}))\mathcal{L}}$ in \eqref{eq1} is \[g^{1}_{\mathcal{L}+(\mathrm{Id}\otimes\tau_{2}(\varpi_{D}))\mathcal{L}}\oplus\bigoplus_{\lambda\in \mathbb{F}_{q^{d}}}g^{su_{\lambda}}_{\mathcal{L}+(\mathrm{Id}\otimes\tau_{2}(\varpi_{D}))\mathcal{L}}.\] To describe the second summand $\mathfrak{K}_{0}\cdot f^{s}_{\mathcal{L}+(\tau_{1}(\varpi_{D})\otimes\mathrm{Id})\mathcal{L}}$, let
	\[F^{x}_{v}:=u_{x}s\cdot f^{s}_{v} \hspace{3mm} \text{for $x\in \mathbb{F}_{q^{d}}$ and $v\in\mathcal{L}+(\tau_{1}(\varpi_{D})\otimes\mathrm{Id})\mathcal{L}$}.\]
	The lattice $f^{s}_{\mathcal{L}+(\tau_{1}(\varpi_{D})\otimes\mathrm{Id})\mathcal{L}}$ is stable under the action of $I\varpi_{D}^{\mathbb{Z}}$. The set $\lbrace 1, u_{x}s:x\in\mathbb{F}_{q^{d}}\rbrace$ forms a set of representatives of $\mathfrak{K}_{0}/I\varpi_{D}^{\mathbb{Z}}$. Thus \begin{equation}\label{eq2}
	    \mathfrak{K}_{0}\cdot f^{s}_{\mathcal{L}+(\tau_{1}(\varpi_{D})\otimes\mathrm{Id})\mathcal{L}}=f^{s}_{\mathcal{L}+(\tau_{1}(\varpi_{D})\otimes\mathrm{Id})\mathcal{L}}+\langle\sideset{}{'}\sum_{x,v}F^{x}_{v}\rangle,
	\end{equation} where $\sum'_{x,v}$ denotes a sum over finitely many pairs $(x,v)$ with $x\in \mathbb{F}_{q^{d}}$ and  $v\in\mathcal{L}+(\tau_{1}(\varpi_{D})\otimes\mathrm{Id})\mathcal{L}$ and $\langle\sideset{}{'}\sum_{x,v}F^{x}_{v}\rangle$ is the $\mathcal{O}$-module generated by all such sums.
	Using $su_{c}s=\left(\begin{smallmatrix} -1/[c] & 1 \\ 0 & [c]\end{smallmatrix}\right)su_{1/c}$ for $c\neq 0$, we get \[F^{x}_{v}=u_{x}s\cdot\sum_{\lambda\in \mathbb{F}_{q^{d}}}g^{su_{\lambda}}_{v}=g^{1}_{v}+\omega_{\tau_{1}}(-1)\sum_{\lambda\in \mathbb{F}_{q^{d}}^{\times}}g^{su_{(1/\lambda)-x}}_{\xi_{\lambda}(v)}\] where $\xi_{\lambda}=\tau_{1}(\lambda)\otimes\tau_{2}(1/\lambda)$. For a fixed but arbitrary $v\in\mathcal{L}+(\tau_{1}(\varpi_{D})\otimes\mathrm{Id})\mathcal{L}$, consider \begin{align*}
	    \sum_{x\in \mathbb{F}_{q^{d}}}F^{x}_{v}&=q^{d}g^{1}_{v}+\omega_{\tau_{1}}(-1)\sum_{x\in \mathbb{F}_{q^{d}}}\sum_{\lambda\in \mathbb{F}_{q^{d}}^{\times}}g^{su_{(1/\lambda)-x}}_{\xi_{\lambda}(v)}\\
	    &=f^{1}_{q^{d}v}+\omega_{\tau_{1}}(-1)\sum_{\lambda\in \mathbb{F}_{q^{d}}^{\times}}\sum_{x\in \mathbb{F}_{q^{d}}}g^{su_{(1/\lambda)-x}}_{\xi_{\lambda}(v)}\\
	    &=f^{1}_{q^{d}v}+\omega_{\tau_{1}}(-1)\sum_{\lambda\in \mathbb{F}_{q^{d}}^{\times}}f^{s}_{\xi_{\lambda}(v)}.
	\end{align*}
	Since $\xi_{\lambda}(v)\in\mathcal{L}+(\tau_{1}(\varpi_{D})\otimes\mathrm{Id})\mathcal{L}$, it follows that $f^{1}_{q^{d}v}\in\mathfrak{K}_{0}\cdot f^{s}_{\mathcal{L}+(\tau_{1}(\varpi_{D})\otimes\mathrm{Id})\mathcal{L}}$. Therefore, we have \[f^{1}_{q^{d}(\tau_{1}(\varpi_{D})\otimes\mathrm{Id})\mathcal{L}}\subseteq (\mathfrak{K}_{0}\cdot L_{1})^{I(1)}\subseteq L_{1}^{(1)},\hspace{1mm}\text{and thus}\hspace{2mm}f^{s}_{q^{d}(\tau_{1}(\varpi_{D})\otimes\mathrm{Id})\mathcal{L}}\subseteq (\mathfrak{K}_{0}\cdot L_{1}^{(1)})^{I(1)}.\] Letting $\mathcal{L}^{(i)}_{1}=(\mathrm{Id}\otimes\tau_{2}(\varpi_{D}))^{i}\mathcal{L}$ and $\mathcal{L}^{(i)}_{2}=(q^{d}(\tau_{1}(\varpi_{D})\otimes\mathrm{Id}))^{i}\mathcal{L}$, we have \[f^{1}_{\mathcal{L}^{(1)}_{1}}\oplus f^{s}_{\mathcal{L}^{(1)}_{2}}\subseteq (\mathfrak{K}_{0}\cdot L_{1}^{(1)})^{I(1)}.\] Hence, as above \[f^{1}_{\mathcal{L}^{(1)}_{1}+(\mathrm{Id}\otimes\tau_{2}(\varpi_{D}))\mathcal{L}^{(1)}_{1}}\oplus f^{s}_{\mathcal{L}^{(1)}_{2}+(\tau_{1}(\varpi_{D})\otimes\mathrm{Id})\mathcal{L}^{(1)}_{2}}\subseteq L_{1}^{(2)}.\] Then, by repeating the above argument, we find $f^{1}_{\mathcal{L}^{(2)}_{1}}\subseteq L_{1}^{(3)}$ and $f^{1}_{\mathcal{L}^{(2)}_{2}}\subseteq L_{1}^{(3)}$. Repeatedly using this argument gives, in general,
	\[f^{1}_{\mathcal{L}^{(i)}_{1}}\subseteq L_{1}^{(2i-1)}\hspace{2mm}\text{and}\hspace{2mm}f^{1}_{\mathcal{L}^{(i)}_{2}}\subseteq L_{1}^{(2i-1)}\hspace{2mm}\text{for all $i\geq 0$}.\]
	By Remark \ref{zigzag_rmk}, the existence of an integral structure in $\pi$ implies that the sequence of $\mathfrak{K}_{1}$-lattices $L_{1}^{(i)}$ stabilizes. This implies that the linear maps $\mathrm{Id}\otimes\tau_{2}(\varpi_{D})$ and $q^{d}(\tau_{1}(\varpi_{D})\otimes\mathrm{Id})$ stabilize some lattices in $W_{1}\otimes W_{2}$. Taking the $d$-th power of these maps, we get $\omega_{\tau_{2}}(\varpi_{F})\in\mathcal{O}$ and $q^{d^{2}}\omega_{\tau_{1}}(\varpi_{F})\in\mathcal{O}$.  
	
	$(\impliedby)$
	The assumptions on $\pi$ that its central character is integral and \[ \omega_{\tau_{2}}(\varpi_{F}),q^{d^{2}}\omega_{\tau_{1}}(\varpi_{F})\in\mathcal{O}\] imply that there exists a $T_{0}$-lattice $\mathcal{L}\subseteq W_{1}\otimes W_{2}$ such that \begin{equation}\label{eq3}
          (\mathrm{Id}\otimes\tau_{2}(\varpi_{D}))\mathcal{L}\subseteq\mathcal{L}
          \hspace{2mm}\text{and}
          \hspace{2mm}
          q^{d}(\tau_{1}(\varpi_{D})\otimes\mathrm{Id})\mathcal{L}\subseteq\mathcal{L}.    
	\end{equation}
	Using $\mathcal{L}$, we define the lattices $L_{0}$ and $L_{1}$ as before. Because of \eqref{eq3},
	\[L_{1}=f^{1}_{\mathcal{L}}\oplus f^{s}_{\mathcal{L}+(\tau_{1}(\varpi_{D})\otimes\mathrm{Id})\mathcal{L}}.\] We will prove that $L^{(1)}_{1}=L^{(0)}_{1}=L_{1}$ which implies the integrality of $\pi$ by Corollary \ref{zigzag_criterion}. This is equivalent to proving that  \[(\mathfrak{K}_{0}\cdot L_{1})\cap V_{1}=(\mathfrak{K}_{0}\cdot L_{1})^{I(1)}=L_{1}.\] By \eqref{eq1} and \eqref{eq2}, it is enough to show that if \[\text{$l+\sideset{}{'}\sum_{x,v}F^{x}_{v}\in(\mathfrak{K}_{0}\cdot L_{1})^{I(1)}$ with $l\in L_{0}$,}\] then $l+\sum'_{x,v}F^{x}_{v}\in L_{1}$. Since $q^{d}-1$ is invertible in $\mathcal{O}$, we can choose an $\mathcal{O}$-basis $\{v_{1},\ldots,v_{n}\}$ of the lattice $\mathcal{L}+(\tau_{1}(\varpi_{D})\otimes\mathrm{Id})\mathcal{L}$ that is an eigenbasis for the operators $\xi_{\lambda}$ and such that some suitable scalar multiples of the $v_{i}$'s form an $\mathcal{O}$-basis of the sub-lattice $\mathcal{L}$. Let $\psi_{i}(\lambda)$ be the eigenvalue for the action of $\xi_{1/\lambda}$ on $v_{i}$. Then $\psi_{i}$ defines a character on $\mathbb{F}_{q^{d}}^{\times}$. Extend $\psi_{i}$ to a function on $\mathbb{F}_{q^{d}}$ by defining $\psi_{i}(0)=0$. 
	
	Let $l+\sideset{}{'}\sum_{x,v}F^{x}_{v}\in(\mathfrak{K}_{0}\cdot L_{1})^{I(1)}$ with $l\in L_{0}$. We write \begin{align*}
	    \sideset{}{'}\sum_{x,v}F^{x}_{v}&=\sum_{x\in \mathbb{F}_{q^{d}}}\left(\sum_{i=1}^{n}a_{x,i}F^{x}_{v_{i}}\right)\\&=\sum_{i=1}^{n}\sum_{x\in \mathbb{F}_{q^{d}}}a_{x,i}F^{x}_{v_{i}}\\&=\sum_{i=1}^{n}\left(\left(\sum_{x\in \mathbb{F}_{q^{d}}}a_{x,i}\right) g^{1}_{v_{i}}+\omega_{\tau_{1}}(-1)\sum_{x\in \mathbb{F}_{q^{d}}}a_{x,i}\sum_{\lambda\in \mathbb{F}_{q^{d}}^{\times}}g^{su_{(1/\lambda)-x}}_{\xi_{\lambda}(v_{i})}\right)\\
	    &=\sum_{i=1}^{n}\left(f^{1}_{\left(\sum_{x\in \mathbb{F}_{q^{d}}}a_{x,i}\right)(v_{i})}+\omega_{\tau_{1}}(-1)\sum_{x,\lambda\in \mathbb{F}_{q^{d}}}a_{x,i}\psi_{i}(\lambda)g^{su_{\lambda-x}}_{v_{i}}\right).
	\end{align*}
	Let us write $a_{x,i}=a_{i}(-x)$ to view it as a function on $\mathbb{F}_{q^{d}}$, and let \[S_{1}=\sum_{i=1}^{n}f^{1}_{\left(\sum_{x\in \mathbb{F}_{q^{d}}}a_{i}(-x)\right)(v_{i})}\hspace{2mm}\text{and}\hspace{2mm}S_{2}=\omega_{\tau_{1}}(-1)\sum_{i=1}^{n}\sum_{x,\lambda\in \mathbb{F}_{q^{d}}}a_{i}(-x)\psi_{i}(\lambda)g^{su_{\lambda-x}}_{v_{i}}.\] Thus \[\sideset{}{'}\sum_{x,v}F^{x}_{v}=S_{1}+S_{2}.\] We may take $l\in\bigoplus_{\lambda\in\mathbb{F}_{q^{d}}}g^{su_{\lambda}}_{\mathcal{L}}$ and thus write \[l=\sum_{i=1}^{n}\sum_{\lambda\in \mathbb{F}_{q^{d}}}b_{i}(\lambda)g^{su_{\lambda}}_{v_{i}}.\] Recall that \[V_{1}=f^{1}_{W_{1}\otimes W_{2}}\oplus f^{s}_{W_{1}\otimes W_{2}}\hspace{2mm}\text{and}\hspace{2mm} L_{1}=f^{1}_{\mathcal{L}}\oplus f^{s}_{\mathcal{L}+(\tau_{1}(\varpi_{D})\otimes\mathrm{Id})\mathcal{L}}.\] Note that $S_{1}$ is invariant under the action of $I(1)$. Therefore, $l+S_{2}$ is also invariant under the action of $I(1)$. Further, the function $l+S_{2}$ is \emph{not} supported on $BI(1)$. Hence $l+S_{2}\in f^{s}_{W_{1}\otimes W_{2}}$. Writing
	\[l+S_{2}=\sum_{y\in \mathbb{F}_{q^{d}}}c_{1}g^{su_{y}}_{v_{1}}+\ldots+\sum_{y\in \mathbb{F}_{q^{d}}}c_{n}g^{su_{y}}_{v_{n}}\] gives that the function \[y\mapsto\omega_{\tau_{1}}(-1)\left(\sum_{x\in \mathbb{F}_{q^{d}}}a_{i}(-x)\psi_{i}(x+y)\right)+b_{i}(y)\hspace{3mm}\text{on $\mathbb{F}_{q^{d}}$}\] is the constant function $c_{i}$ for all $1\leq i\leq n$. Thus $c_{i}\in\mathcal{O}$ for all $i$, and $l+S_{2}=f^{s}_{\sum_{i}c_{i}v_{i}}\in f^{s}_{\mathcal{L}+(\tau_{1}(\varpi_{D})\otimes\mathrm{Id})\mathcal{L}}\subseteq L_{1}$. 
	
	To show that $S_{1}\in L_{1}$, we use Fourier theoretic methods as in \cite[\S 3]{vig08}. We assume that our coefficient field $E$ contains $p$-roots of unity so that there exists a non-trivial additive character $\eta$ on $\mathbb{F}_{q^{d}}$ and the Fourier transform of $\mathcal{O}$-valued functions on $\mathbb{F}_{q^{d}}$ is well-defined. Given a function $f:\mathbb{F}_{q^{d}}\rightarrow\mathcal{O}$, its Fourier transform $\widehat{f}$ is defined to be  \[\widehat{f}(x)=\sum_{y\in\mathbb{F}_{q^{d}}}\eta(xy)f(y).\] Given two functions $f,g:\mathbb{F}_{q^{d}}\rightarrow\mathcal{O}$, their convolution product $f\ast g$ is the function $(f\ast g)(x)=\sum_{x=y+z}f(y)g(z)$. One has $\widehat{\widehat{f}}=q^{d}f$ and $\widehat{f\ast g}=\widehat{f}\widehat{g}$. Let $\Delta$ denote the constant function $1$ on $\mathbb{F}_{q^{d}}$ and $\delta_{0}$ denote the characteristic function of $0$. Then $\widehat{\Delta}=q^{d}\delta_{0}$. In this language, we have $S_{1}=f^{1}_{\sum_{i=1}^{n}\widehat{a_{i}}(0)v_{i}}$. We show that $\widehat{a_{i}}(0)v_{i}\in\mathcal{L}$ for each $1\leq i\leq n$. Indeed, for each $i$, we have from the previous paragraph \begin{equation}\label{eqn4}
		c_{i}\Delta=\omega_{\tau_{1}}(-1)(a_{i}\ast\psi_{i})+b_{i}.
	\end{equation} If $\psi_{i}$ is trivial, then it follows that $c_{i}=\omega_{\tau_{1}}(-1)(\widehat{a_{i}}(0)-a_{i}(y))+b_{i}(y)$ for all $y\in\mathbb{F}_{q^{d}}$. Hence \[(\widehat{a_{i}}(0)-\omega_{\tau_{1}}(-1)c_{i})v_{i}+\omega_{\tau_{1}}(-1)b_{i}(y)v_{i}=a_{i}(y)v_{i}\hspace{2mm}\text{for all $y\in\mathbb{F}_{q^{d}}$}.\] Adding over $y\in\mathbb{F}_{q^{d}}$ gives \[(\widehat{a_{i}}(0)-\omega_{\tau_{1}}(-1)c_{i})q^{d}v_{i}+\omega_{\tau_{1}}(-1)\widehat{b_{i}}(0)v_{i}=\widehat{a_{i}}(0)v_{i}.\] By \eqref{eq3}, $(\widehat{a_{i}}(0)-\omega_{\tau_{1}}(-1)c_{i})q^{d}v_{i}\in\mathcal{L}$. Further, recall that $\sum_{i=1}^{n}b_{i}(y)v_{i}\in\mathcal{L}$. The choice of the basis $\{v_{i}\}$ implies that each $b_{i}(y)v_{i}\in\mathcal{L}$. Thus $\widehat{b_{i}}(0)v_{i}\in\mathcal{L}$. Hence $\widehat{a_{i}}(0)v_{i}\in\mathcal{L}$. If $\psi_{i}$ is non-trivial, then applying the Fourier transform on both sides of \eqref{eqn4} gives \[c_{i}q^{d}\delta_{0}=\omega_{\tau_{1}}(-1)\widehat{a_{i}}\widehat{\psi_{i}}+\widehat{b_{i}}.\] 
	Multiplying both sides by $\widehat{\psi_{i}^{-1}}$ and using that $\widehat{\psi_{i}}\widehat{\psi_{i}^{-1}}=\psi_{i}(-1)q^{d}(\Delta-\delta_{0})$, we get \[c_{i}q^{d}\delta_{0}\widehat{\psi_{i}^{-1}}=\omega_{\tau_{1}}(-1)\psi_{i}(-1)q^{d}(\Delta-\delta_{0})\widehat{a_{i}}+\widehat{b_{i}}\widehat{\psi_{i}^{-1}}.\] Thus 
	\[c_{i}q^{d}\delta_{0}\widehat{\psi_{i}^{-1}}=\omega_{\tau_{1}}(-1)\psi_{i}(-1)q^{d}\widehat{a_{i}}-\omega_{\tau_{1}}(-1)\psi_{i}(-1)q^{d}\delta_{0}\widehat{a_{i}}+\widehat{b_{i}}\widehat{\psi_{i}^{-1}}.\] Rearranging the terms, we have 
	\[q^{d}\widehat{a_{i}}=q^{d}\widehat{a_{i}}(0)\delta_{0}+\widehat{\psi_{i}^{-1}}\widehat{\phi}=q^{d}\widehat{a_{i}}(0)\delta_{0}+\widehat{{\psi_{i}^{-1}}\ast\phi}\] where $\phi=\omega_{\tau_{1}}(-1)\psi_{i}(-1)(c_{i}\Delta-b_{i})$. By the Fourier transform again, we get \[q^{d}a_{i}=\widehat{a_{i}}(0)\Delta+\psi_{i}^{-1}\ast\phi.\] Hence \[\widehat{a_{i}}(0)v_{i}=a_{i}(0)q^{d}v_{i}-(\psi_{i}^{-1}\ast\phi)(0)v_{i}.\] Note that \begin{align*}
		(\psi_{i}^{-1}\ast\phi)(0)v_{i}&=\omega_{\tau_{1}}(-1)\psi_{i}(-1)\left(\sum_{x\in \mathbb{F}_{q^{d}}}\psi_{i}^{-1}(-x)(c_{i}-b_{i}(x))\right)v_{i}\\&=\omega_{\tau_{1}}(-1)\psi_{i}(-1)\left(c_{i}\widehat{\psi_{i}^{-1}}(0)-\sum_{x\in \mathbb{F}_{q^{d}}}\psi_{i}^{-1}(-x)b_{i}(x)\right)v_{i}\\&=-\omega_{\tau_{1}}(-1)\psi_{i}(-1)\sum_{x\in \mathbb{F}_{q^{d}}}\psi_{i}^{-1}(-x)b_{i}(x)v_{i}\in\mathcal{L}.
	\end{align*} Also $q^{d}v_{i}\in\mathcal{L}$ by \eqref{eq3}. Therefore $\widehat{a_{i}}(0)v_{i}\in \mathcal{L}$. 
	
	It follows that  $S_{1}=f^{1}_{\sum_{i=1}^{n}\widehat{a_{i}}(0)v_{i}}\in f^{1}_{\mathcal{L}}\subseteq L_{1}$. Therefore, $l+\sum'_{x,v}F^{x}_{v}=l+S_{1}+S_{2}\in L_{1}$.
	\end{proof}
	
	Let $E$ be large enough to contain $\sqrt{q}$, and let \[\tau_{1}\times\tau_{2}:=\mathrm{Ind}_{B}^{G}\left(\tau_{1}|\cdot|^\frac{1}{2}\otimes\tau_{2}|\cdot|^{-\frac{1}{2}}\right)\] be the normalized parabolic induction over $E$. The integrality criterion in Theorem \ref{main_thm_for_prin_ser} is symmetric for the normalized induction:
	
	\begin{theorem}\label{crit_norm_para_ind}
	Let $\tau_{1}$ and $\tau_{2}$ be smooth irreducible tamely ramified representations of $D^{\times}$. The representation $\tau_{1}\times\tau_{2}$ with integral central character admits an integral structure if and only if $\omega_{\tau_{1}}(\varpi_{F})q^{\frac{d^{2}}2},\omega_{\tau_{2}}(\varpi_{F})q^{\frac{d^{2}}2}\in\mathcal{O}$. 
	\end{theorem}
	
	As a consequence of Theorem \ref{crit_norm_para_ind}, we obtain that when $\tau_{1}\times\tau_{2}$ with integral central character is reducible, its irreducible subquotients are always integral. Indeed, enlarging $E$ if necessary, we assume that all irreducible subquotients of $\tau_{1}\times\tau_{2}$ are defined over $E$. We recall a result of Tadi\'{c} which says that $(\tau_{1}\times\tau_{2})_{\mathbb{C}}$, which is isomorphic to $(\tau_{1})_{\mathbb{C}}\times(\tau_{2})_{\mathbb{C}}$, is reducible if and only if $(\tau_{2})_{\mathbb{C}}\cong(\tau_{1})_{\mathbb{C}}|\cdot|^{\pm\frac{a}d}_{\mathbb{C}}$, and in this case, it is multiplicity-free and has length $2$ \cite[Lemmas 2.5, 4.2, Proposition 4.3]{tad90}. Here $a$ is the length of the segment of the essentially square-integrable representation of $\mathrm{GL}_{d}(F)$ associated to $(\tau_{1})_{\mathbb{C}}$ under the Jacquet-Langlands correspondence. It follows that $\tau_{1}\times\tau_{2}$ is reducible over $E$ if and only if $\tau_{2}\cong\tau_{1}|\cdot|^{\pm\frac{a}{d}}$ as $E$-linear representations (again after enlarging $E$ if necessary). If $\tau_{1}$ is tamely ramified of dimension  $d'$, then $d=ad'$.
	
	Let $\tau=\tau_{1}|\cdot|^{\frac{a}{2d}}$. Denoting by $\mathrm{St}(\tau)$ and $\mathrm{Sp}(\tau)$ the $E$-models of the \emph{Steinberg} $\mathrm{St}(\tau_{\mathbb{C}})$ and the \emph{Speh} $\mathrm{Sp}(\tau_{\mathbb{C}})$ representation respectively, one has the following short exact sequence of smooth $E$-linear representations 
	\begin{equation}\label{spstseq}
	0\longrightarrow\mathrm{Sp}(\tau)\longrightarrow\tau|\cdot|^{-\frac{a}{2d}}\times\tau|\cdot|^{\frac{a}{2d}}\longrightarrow\mathrm{St}(\tau)\longrightarrow 0
	\end{equation} (see \cite[Theorem 2.2]{rag07}).
	We remark that $\mathrm{Sp}(\tau)$ is infinite-dimensional if and only if $\tau$ has dimension $>1$ \cite[Remark 2.4]{rag07}. 
	
	\begin{theorem}\label{main_thm_for_st_and_sp}
	Let $\tau$ be a smooth absolutely irreducible tamely ramified representation of $D^{\times}$ over $E$. The representation $\mathrm{St}(\tau)$ with integral central character admits an integral structure. The representation $\mathrm{Sp}(\tau)$ with integral central character admits an integral structure.
	\end{theorem}
	
	\begin{proof}
	Let $\Pi=\tau|\cdot|^{-\frac{a}{2d}}\times\tau|\cdot|^{\frac{a}{2d}}$. 
	From the sequence \eqref{spstseq}, we see that $\omega_{\mathrm{Sp}(\tau)}=\omega_{\mathrm{St}(\tau)}=\omega_{\Pi}=\omega_{\tau}^{2}$. If $\omega_{\mathrm{Sp}(\tau)}$ is integral, then $\omega_{\tau}(\varpi_{F})\in\mathcal{O}^{\times}$ and thus $\omega_{\Pi}$ is integral. Further, note that \[\text{$\omega_{\tau}(\varpi_{F})q^{\frac{d(d+a)}{2}},\omega_{\tau}(\varpi_{F})q^{\frac{d(d-a)}{2}}\in\mathcal{O}$ because $d+a\geq 0$ and $d-a\geq 0$.}\] Hence, by Theorem \ref{crit_norm_para_ind}, $\Pi$ has an integral structure and thus its irreducible subrepresentation $\mathrm{Sp}(\tau)$ also has an integral structure. One similarly shows that $\mathrm{St}(\tau)$ with integral central character $\omega_{\mathrm{St}(\tau)}$ has an integral structure by considering the exact sequence dual to \eqref{spstseq}. 
	\end{proof}
	
	\begin{corollary} Let $\tau$ be a smooth absolutely irreducible tamely ramified representation of $D^{\times}$ over $E$. Then $\mathrm{St}(\tau)$ is integral if and only if $\mathrm{Sp}(\tau)$ is integral.\hspace{4.85cm}\qedsymbol
	\end{corollary}
	
	\section{Integrality of locally algebraic representations}\label{section4}
	
	In this section, $\pi=\pi_{sm}\otimes\pi_{alg}$ where $\pi_{alg}$ is a non-trivial irreducible algebraic representation of $G$  over $E$. We begin with a simple generalization of \cite[Proposition 2.2]{hu20} of Hu on diagrams of $k$-vector spaces with trivial $0$-th homology.
	We say that a diagram $\mathcal{D}_{1}\xrightarrow{r}\mathcal{D}_{0}$ admits a central character if $Z$ acts on $\mathcal{D}_{0}$ and $\mathcal{D}_{1}$ by a character.
	
	\begin{lemma}\label{h0diagrams}
	Let $\mathcal{D}$ be a diagram $\mathcal{D}_{1}\xrightarrow{r}\mathcal{D}_{0}$ of smooth $k$-representations admitting a central character such that $H_{0}(\mathcal{D})=0$ and $\mathcal{D}_{1}$ is an irreducible representation of $\mathfrak{K}_{1}$. Then \[\mathrm{dim}_{k}\mathcal{D}_{0}\leq\frac{q^{d}+1}{2}\mathrm{dim}_{k}\mathcal{D}_{1}.\] Moreover, if $\mathrm{dim}_{k}\mathcal{D}_{0}=\frac{q^{d}+1}{2}\mathrm{dim}_{k}\mathcal{D}_{1}$, then $\mathcal{D}_{0}\cong\mathrm{Ind}_{I\varpi_{D}^{\mathbb{Z}}}^{\mathfrak{K}_{0}}r(\mathcal{D}_{1})$.
	\end{lemma}
	
	\begin{proof}
	Since $H_{0}(\mathcal{D})=0$, $\mathrm{Ker}(r)\neq 0$. Pick a non-zero $I/I(1)$-eigenvector $v\in\mathrm{Ker}(r)$. Let $\mathcal{D}'\subseteq\mathcal{D}$ be the subdiagram $\mathcal{D}'_{1}\xrightarrow{r'}\mathcal{D}'_{0}$ where \begin{equation}\label{gradedpieces}
	    \mathcal{D}'_{1}=\mathfrak{K}_{1}\cdot v, \hspace{1mm} \mathcal{D}_{0}'=\mathfrak{K}_{0}\cdot r(\mathcal{D}_{1}'), \hspace{1mm} \mathrm{and} \hspace{1mm} r'=r|_{\mathcal{D}_{1}'}.
	\end{equation} Since $\mathcal{D}_{1}$ is irreducible, $\mathcal{D}_{1}/\mathcal{D}_{1}'=0$. Further, $H_{0}(\mathcal{D})=0$ implies that $H_{0}(\mathcal{D}/\mathcal{D}')=0$. Therefore, $\mathcal{D}_{0}/\mathcal{D}_{0}'=0$. Consequently, $\mathcal{D}'=\mathcal{D}$, and thus there is a surjection \begin{equation}\label{surjection}
	    \mathrm{Ind}_{I\varpi_{D}^{\mathbb{Z}}}^{\mathfrak{K}_{0}}r(\mathcal{D}_{1})\twoheadrightarrow\mathcal{D}_{0}.
	\end{equation} 
	Let $\mathcal{D}_{1}^{0}\subseteq\mathcal{D}_{1}$ be the $k$-span of the vectors $\varpi_{D}^{i}v$ for $0\leq i\leq d-1$. Then $\mathcal{D}_{1}=\mathcal{D}_{1}^{0}+t\mathcal{D}_{1}^{0}$. Note that $t$ is a linear isomorphism. So \[\frac{\mathrm{dim}_{k}\mathcal{D}_{1}}{2}\leq \mathrm{dim}_{k}\mathcal{D}_{1}^{0}.\] Moreover, as $r(v)=0$ and $r$ is $\varpi_{D}^{\mathbb{Z}}$-linear, we have $\mathcal{D}_{1}^{0}\subseteq\mathrm{Ker}(r)$. Thus \[\frac{\mathrm{dim}_{k}\mathcal{D}_{1}}{2}\leq\mathrm{dim}_{k}\mathcal{D}_{1}^{0}\leq \mathrm{dim}_{k}\mathrm{Ker}(r).\] 
	Therefore, $\mathrm{dim}_{k}\mathcal{D}_{1}=\mathrm{dim}_{k}r(\mathcal{D}_{1})+\mathrm{dim}_{k}\mathrm{Ker}(r)\geq\mathrm{dim}_{k}r(\mathcal{D}_{1})+\frac{\mathrm{dim}_{k}\mathcal{D}_{1}}{2}$. Hence $\mathrm{dim}_{k}r(\mathcal{D}_{1})\leq\frac{\mathrm{dim}_{k}\mathcal{D}_{1}}{2}$. Now it follows from \eqref{surjection} that 
	\[\mathrm{dim}_{k}\mathcal{D}_{0}\leq\frac{q^{d}+1}{2}\mathrm{dim}_{k}\mathcal{D}_{1}\]
	because $[\mathfrak{K}_{0}:I\varpi_{D}^{\mathbb{Z}}]=q^{d}+1$.
	
	If $\mathrm{dim}_{k}\mathcal{D}_{0}=\frac{q^{d}+1}{2}\mathrm{dim}_{k}\mathcal{D}_{1}$, then from \eqref{surjection}, we get that $\frac{\mathrm{dim}_{k}\mathcal{D}_{1}}{2}\leq\mathrm{dim}_{k}r(\mathcal{D}_{1})$. By the previous paragraph, this implies $\frac{\mathrm{dim}_{k}\mathcal{D}_{1}}{2}=\mathrm{dim}_{k}r(\mathcal{D}_{1})$ and thus $\mathcal{D}_{0}\cong\mathrm{Ind}_{I\varpi_{D}^{\mathbb{Z}}}^{\mathfrak{K}_{0}}r(\mathcal{D}_{1})$.
	\end{proof}
	
	\begin{remark}\label{rmkh0diagrams}
	In the above lemma, if $\mathcal{D}_{1}$ is not irreducible, then the diagram $\mathcal{D}$ has a filtration by subdiagrams whose graded pieces are the diagrams of the form \eqref{gradedpieces}. Thus the same dimension relation as in the lemma holds if $\mathrm{dim}_{k}\mathcal{D}_{1}<\infty$. Further, if $\mathrm{dim}_{k}\mathcal{D}_{0}=\frac{q^{d}+1}{2}\mathrm{dim}_{k}\mathcal{D}_{1}$, then for any graded piece $\mathcal{Q}_{1}\xrightarrow{\overline{r}}\mathcal{Q}_{0}$ of $\mathcal{D}$, $\mathcal{Q}_{0}\cong\mathrm{Ind}_{I\varpi_{D}^{\mathbb{Z}}}^{\mathfrak{K}_{0}}\overline{r}(\mathcal{Q}_{1})$.
	\end{remark}
	
	The next lemma is well-known: 
	\begin{lemma}\label{finite_homothey_classes_lemma}
	    Suppose a group $H$ acting on a finite-dimensional $\mathbb{Q}_{p}$-vector space $V$ stabilizes a lattice in $V$. Then $H$ stabilizes finitely many homothety classes of lattices in $V$ if and only if its action on $V$ is irreducible.
	\end{lemma}
	\begin{proof}
	The group $H$ acts irreducibly on $V$ if and only if its image $\overline{H}$ in $\mathrm{GL}(V)$ is not contained in a proper parabolic subgroup of $\mathrm{GL}(V)\cong\mathrm{GL}_{n}(\mathbb{Q}_{p})$. Suppose $\overline{H}$ is contained in a proper parabolic subgroup of $\mathrm{GL}_{n}(\mathbb{Q}_{p})$. Since $H$ stabilizes a lattice, without loss of generality, we may assume that $\overline{H}$ is a subgroup of a standard proper parabolic subgroup of $\mathrm{GL}_{n}(\mathbb{Z}_{p})$ corresponding to the partition $n=n_{1}+\ldots+n_{k}$. For $m\in\mathbb{N}$, consider the lattice $L_{m}$ given by the direct sum of $n_{1}$ copies of $\frac{1}{p^{m}}\mathbb{Z}_{p}$ and $n_{2}+\ldots+n_{k}$ copies of $\mathbb{Z}_{p}$. Then $H$ stabilizes the infinite family $\lbrace[L_{m}]\rbrace_{m\in\mathbb{N}}$ of homothety classes of lattices.
	
	To prove the converse, fix a set $\{g_\alpha\}$ of coset representatives for ${\rm GL}_n(\mathbb{Q}_p)/\mathbb{Q}_p^\times{\rm GL}_n(\mathbb{Z}_p)$ such that $g_{\alpha}=(g_{ij}^\alpha)\in\mathrm{M}_{n}(\mathbb{Z}_{p})$. If $L_{0}=\mathbb{Z}_{p}\oplus\ldots\oplus\mathbb{Z}_{p}$ denotes the standard lattice in $\mathbb{Q}_{p}^{n}\cong V$, then 
	\[g_\alpha L_{0}=\bigoplus_{i=1}^n p^{\min_j\{\mathrm{val}_{p}(g_{ij}^\alpha)\}}\mathbb{Z}_p.\]
	Suppose $H$ stabilizes a family $\{g_{\alpha}L_{0}:\alpha\in\mathcal{I}\}$ of lattices where $\mathcal{I}$ is not finite. Then there exists $i$, $1\leq i\leq n$, such that \[\max_{\alpha\in \mathcal{I}}\{\min_{j}\{\mathrm{val}_{p}(g_{ij}^\alpha)\}\}\] is not bounded. Consequently, we find that $\bigcap_{\alpha\in \mathcal{I}} g_\alpha L_{0}$ is contained in a proper subspace of $V$ stable under the action of $H$, i.e., the $H$-action on $V$ is reducible.
    \end{proof}

	Let $\tau=\mathrm{Ind}^{D^{\times}}_{D(1)D_{d'}^{\times}}\theta$ be a smooth absolutely irreducible tamely ramified representation of $D^{\times}$ of dimension $d'$. 
	
	\begin{lemma}\label{I(1)-invariants of St and Sp} 
	(i) As $I/I(1)$-representations, \[\mathrm{St}(\tau)^{I(1)}\cong\mathrm{Sp}(\tau)^{I(1)}\cong(\theta\oplus\theta^{q}\oplus\ldots\oplus\theta^{q^{d'-1}})\otimes(\theta\oplus\theta^{q}\oplus\ldots\oplus\theta^{q^{d'-1}}).\] 
	(ii) The representations $\mathrm{St}(\tau)^{I(1)}$ and $\mathrm{Sp}(\tau)^{I(1)}$ are irreducible as $\mathfrak{K}_{1}$-representations if and only if $d'=1,2$.
	\end{lemma}
	\begin{proof}
	We prove the lemma for $\mathrm{St}(\tau)$; the proof for $\mathrm{Sp}(\tau)$ is similar.
	
	(i) By \cite[Proposition 6.7]{mp96}, the natural $T$-equivariant surjective map \[\mathrm{St}(\tau_{\mathbb{C}})\twoheadrightarrow\mathrm{St}(\tau_{\mathbb{C}})_{N}\] from the Steinberg representation to its smooth Jacquet module induces a $(T\cap I)$-equivariant isomorphism \[\mathrm{St}(\tau_{\mathbb{C}})^{I(1)}\rightarrow(\mathrm{St}(\tau_{\mathbb{C}})_{N})^{T\cap I(1)}.\] By \cite[Theorem 2.2 (ii)]{rag07}, $\mathrm{St}(\tau_{\mathbb{C}})_{N}\cong\tau_{\mathbb{C}}|\cdot|_{\mathbb{C}}^{\frac{a+d}{2d}}\otimes\tau_{\mathbb{C}}|\cdot|_{\mathbb{C}}^{-\frac{a+d}{2d}}$ as $T$-representations. So the group $T\cap I(1)=D(1)\times D(1)$ acts trivially on $\mathrm{St}(\tau_{\mathbb{C}})_{N}$ because $\tau_{\mathbb{C}}$ is tamely ramified, i.e., \[(\mathrm{St}(\tau_{\mathbb{C}})_{N})^{T\cap I(1)}=\mathrm{St}(\tau_{\mathbb{C}})_{N}.\] Since all the representations are defined over $E$, it follows that \[\mathrm{St}(\tau)^{I(1)}\cong\tau|\cdot|^{\frac{a+d}{2d}}\otimes\tau|\cdot|^{-\frac{a+d}{2d}}\] as $E$-linear $I/I(1)$-representations. Now, (i) follows from the isomorphisms \[\tau|\cdot|^{\frac{a+d}{2d}}\cong\tau|\cdot|^{-\frac{a+d}{2d}}\cong\theta\oplus\theta^{q}\oplus\ldots\oplus\theta^{q^{d'-1}}\] as $\mathcal{O}_{D}^{\times}/D(1)$-representations.
	
	(ii) The $I/I(1)$-representation $\mathrm{St}(\tau)^{I(1)}$ is a sum of $d'^{2}$ distinct $I/I(1)$-characters because the Galois conjugates of $\theta$ are distinct. Recall that $\mathfrak{K}_{1}$ is generated by $I$ and $t$. So the $\mathfrak{K}_{1}$-subrepresentation of $\mathrm{St}(\tau)^{I(1)}$ generated by a non-zero $I/I(1)$-eigenvector has dimension at most $2d'$. Hence $\mathrm{St}(\tau)^{I(1)}$ is reducible if $d'>2$. Conversely, if $d'=2$, it is easy to check that any of the four $I/I(1)$-characters in $\mathrm{St}(\tau)^{I(1)}$ generate the whole representation under the action of $t$.
	\end{proof}
	
	By Frobenius reciprocity, there are non-zero maps of $K$-representations  \[\mathrm{Ind}_{I}^{K}(\theta^{q^{i}}\otimes\theta^{q^{j}})\rightarrow\mathrm{St}(\tau)^{K(1)}\hspace{2mm}\text{and}\hspace{2mm}\mathrm{Ind}_{I}^{K}(\theta^{q^{i}}\otimes\theta^{q^{j}})\rightarrow\mathrm{Sp}(\tau)^{K(1)}\]  for all $0\leq i,j\leq d'-1$. The representations $\mathrm{Ind}_{I}^{K}(\theta^{q^{i}}\otimes\theta^{q^{j}})$ are irreducible if $i\neq j$. Thus, \begin{equation}\label{K-irr in St and Sp}
	    \mathrm{Ind}_{I}^{K}(\theta^{q^{i}}\otimes\theta^{q^{j}})\subseteq\mathrm{St}(\tau)^{K(1)}\hspace{2mm}\text{and}\hspace{2mm}\mathrm{Ind}_{I}^{K}(\theta^{q^{i}}\otimes\theta^{q^{j}})\subseteq\mathrm{Sp}(\tau)^{K(1)}
	\end{equation}  for all $0\leq i< j\leq d'-1$. If $i=j$, then $\mathrm{Ind}_{I}^{K}(\theta^{q^{i}}\otimes\theta^{q^{i}})$ is a sum of 2 irreducible subrepresentations  $\theta^{q^{i}}\circ\mathrm{det}(\overline{\hspace{0.5mm}\cdot\hspace{0.5mm}})$ and $\mathrm{st}(\theta^{q^{i}})$. Here, $\mathrm{det}(\overline{\hspace{0.5mm}\cdot\hspace{0.5mm}})$ is the composition of the determinant character of $\mathrm{GL}_{2}(\mathbb{F}_{q^{d}})$ and the natural surjection $K\twoheadrightarrow\mathrm{GL}_{2}(\mathbb{F}_{q^{d}})$.
	
	
	\begin{lemma}\label{K(1)-invariants of St and Sp}
	As $K$-representations \begin{align*}
	&\mathrm{St}(\tau)^{K(1)}\cong\bigoplus_{i< j}\mathrm{Ind}_{I}^{K}(\theta^{q^{i}}\otimes\theta^{q^{j}})\oplus\bigoplus_{i}\mathrm{st}(\theta^{q^{i}})\hspace{2mm}\text{and}\\&\mathrm{Sp}(\tau)^{K(1)}\cong\bigoplus_{i< j}\mathrm{Ind}_{I}^{K}(\theta^{q^{i}}\otimes\theta^{q^{j}})\oplus\bigoplus_{i}(\theta^{q^{i}}\circ\mathrm{det}(\overline{\hspace{0.5mm}\cdot\hspace{0.5mm}})).
	\end{align*}
	\end{lemma}
	\begin{proof}

	It suffices to show that $\theta\circ\mathrm{det}(\overline{\hspace{0.5mm}\cdot\hspace{0.5mm}})\subseteq\mathrm{Sp}(\tau)^{K(1)}$. Indeed, if $\theta\circ\mathrm{det}(\overline{\hspace{0.5mm}\cdot\hspace{0.5mm}})\subseteq\mathrm{Sp}(\tau)^{K(1)}$, then the diagonal action of $\varpi_{D}$ gives $\theta^{q^{i}}\circ\mathrm{det}(\overline{\hspace{0.5mm}\cdot\hspace{0.5mm}})\subseteq\mathrm{Sp}(\tau)^{K(1)}$ for all $i$, and the multiplicity-freeness of $\mathrm{St}(\tau)^{I(1)}$ and $\mathrm{Sp}(\tau)^{I(1)}$ then implies that $\mathrm{st}(\theta^{q^{i}})\subseteq\mathrm{St}(\tau)^{K(1)}$ for all $i$. We may use \cite[Proposition 7.21(1)]{ms14} to conclude that $\theta\circ\mathrm{det}(\overline{\hspace{0.5mm}\cdot\hspace{0.5mm}})\subseteq\mathrm{Sp}(\tau)^{K(1)}$. 
\end{proof}
	
	\begin{theorem}\label{main_thm__for_la_st_and_sp}
	Let $\tau$ be a smooth absolutely irreducible tamely ramified representation of $D^{\times}$ of dimension $d'\leq 2$. Let $\pi=\mathrm{St}(\tau)\otimes\pi_{alg}$ be an irreducible locally algebraic representation with integral central character. Then $\pi$ admits an integral structure.
	\end{theorem}
	\begin{proof}
	Let $V_{1}=\mathrm{St}(\tau)^{I(1)}\otimes\pi_{alg}$ and $V_{0}=\mathrm{St}(\tau)^{K(1)}\otimes\pi_{alg}$. The group $\mathfrak{K}_{1}$ is isomorphic to $I\rtimes t^{\mathbb{Z}}$. Since $I$ is compact, $t^{2d}\in Z$, and $Z$ acts on $V_{1}$ by an integral character, it follows that $V_{1}$ contains a $\mathfrak{K}_{1}$-lattice $L_{1}$. Moreover, as $\mathrm{St}(\tau)^{I(1)}$ is an irreducible $\mathfrak{K}_{1}$-representation by Lemma \ref{I(1)-invariants of St and Sp} (ii) and $\pi_{alg}$ is also an irreducible $\mathfrak{K}_{1}$-representation, $V_{1}$ is an irreducible locally algebraic representation of $\mathfrak{K}_{1}$. Thus $V_{1}$ contains finitely many homothety classes of $\mathfrak{K}_{1}$-lattices by Lemma \ref{finite_homothey_classes_lemma}. 
	
	Suppose $\pi$ is not integral. Then, by Corollary \ref{zigzag_criterion}, the increasing sequence of $\mathfrak{K}_{1}$-lattices $(L_{1}^{(i)})_{i}$ of $V_{1}$ does not become stationary. By the previous paragraph, there is $i_{0}>0$ such that $L_{1}^{(i_{0})}$ and $L_{1}$ are in the same homothety class, i.e., $L_{1}^{(i_{0})}=\varpi^{j}L_{1}$ for some $j<0$. Let $L_{0}=\mathfrak{K}_{0}\cdot L_{1}$ and $L_{0}^{(i_{0})}=\mathfrak{K}_{0}\cdot L_{1}^{(i_{0})}=\varpi^{j}L_{0}$. Let \[\mathcal{D}_{\mathcal{O}}=L_{1}\hookrightarrow L_{0}\hspace{3mm}\text{and}\hspace{3mm}\mathcal{D}_{\mathcal{O}}^{(i_{0})}=L_{1}^{(i_{0})}\hookrightarrow L_{0}^{(i_{0})}\] be the corresponding diagrams of $\mathcal{O}$-free $\mathcal{O}\mathfrak{K}_{i}$-modules. The natural map $H_{0}(\mathcal{D}_{\mathcal{O}})\rightarrow H_{0}(\mathcal{D}_{\mathcal{O}}^{(i_{0})})$ is surjective by Lemma \ref{surjectivity_on_h0}. As the diagram $\mathcal{D}_{\mathcal{O}}^{(i_{0})}$ is equal to the diagram $\varpi^{j}\mathcal{D}_{\mathcal{O}}$, this gives $H_{0}(\varpi^{j}\mathcal{D}_{\mathcal{O}}/\mathcal{D}_{\mathcal{O}})=0$. By d\'{e}vissage, we have $H_{0}(\mathcal{D}_{k})=0$ where  $\mathcal{D}_{k}=\mathcal{D}_{\mathcal{O}}\otimes_{\mathcal{O}}k=\varpi^{-1}\mathcal{D}_{\mathcal{O}}/\mathcal{D}_{\mathcal{O}}$. By Lemma \ref{h0diagrams} and Remark \ref{rmkh0diagrams}, we get that $\mathrm{dim}_{k}(L_{0}\otimes_{\mathcal{O}} k)\leq\frac{q^{d}+1}2\mathrm{dim}_{k}(L_{1}\otimes_{\mathcal{O}} k)$. Since $\mathrm{dim}_{k}(L_{0}\otimes_{\mathcal{O}} k)=\mathrm{dim}_{E}V_{0}$ and $\mathrm{dim}_{k}(L_{1}\otimes_{\mathcal{O}} k)=\mathrm{dim}_{E}V_{1}$, we obtain $\mathrm{dim}_{E}V_{0}\leq\frac{q^{d}+1}2\mathrm{dim}_{E}V_{1}$. However, it follows from Lemma \ref{I(1)-invariants of St and Sp} (i) that \[\mathrm{dim}_{E}\mathrm{St}(\tau)^{I(1)}=d'^{2},\] and from Lemma \ref{K(1)-invariants of St and Sp} that \begin{align*}
	    &\mathrm{dim}_{E}\mathrm{St}(\tau)^{K(1)}=\frac{1}{2}(d'^{2}-d')(q^{d}+1)+d'q^{d}.
	\end{align*}
	This implies that $\mathrm{dim}_{E}V_{0}>\frac{q^{d}+1}{2}\mathrm{dim}_{E}V_{1}$. A contradiction. 
	\end{proof}

	We conclude with an example of an integral locally algebraic Speh representation. For simplicity, we take $D$ to be the quaternionic division algebra. Let $\tau=\mathrm{Ind}^{D^{\times}}_{D(1)D_{2}^{\times}}\theta$ be a smooth absolutely irreducible tamely ramified representation of $D^{\times}$ over $E$. Note that $\tau$ has dimension $2$ and hence $\mathrm{Sp}(\tau)$ is infinite-dimensional. Consider the following irreducible locally algebraic representation \[\pi:=\mathrm{Sp}(\tau)\otimes (\mathrm{Sym}^{1}E^{4}\otimes\mathrm{det}^{-\frac{1}{4}})\] 
	of $G$. Here, $G$ acts on the algebraic representation via $G\hookrightarrow\mathrm{GL}_{4}(F_{2})\hookrightarrow\mathrm{GL}_{4}(E)$ induced by the map $D\rightarrow\mathrm{M}_{2}(F_2)$, $\alpha+\beta\varpi_{D}\mapsto\left(\begin{smallmatrix} \alpha & \beta\varpi_{F} \\ \sigma(\beta) & \sigma(\alpha) \end{smallmatrix}\right)$ where $\sigma$ is the unique non-trivial automorphism in $\mathrm{Gal}(F_2/F)$. We assume that $\omega_{\tau}$ is integral so that the  central character $\omega_{\pi}$ of $\pi$ is integral. We now show that $\pi$ \emph{is} integral which, in particular, implies that Emerton's conditions are sufficient for the integrality of $\pi$. We take $s=\left(\begin{smallmatrix}0 & 1\\-1 & 0\end{smallmatrix}\right)$ so that $\mathrm{det}(\overline{s})=1$.
	
	\begin{theorem}\label{la_sp_int}
	The representation $\pi=\mathrm{Sp}(\tau)\otimes (\mathrm{Sym}^{1}E^{4}\otimes\mathrm{det}^{-\frac{1}{4}})$ admits an integral structure.
	\end{theorem}
	\begin{proof}
	Recall from Lemma \ref{K(1)-invariants of St and Sp} that \begin{align*}
	    &\mathrm{Sp}(\tau)^{K(1)}=(\theta\circ\mathrm{det}(\overline{\hspace{0.5mm}\cdot\hspace{0.5mm}}))\oplus(\theta^{q}\circ\mathrm{det}(\overline{\hspace{0.5mm}\cdot\hspace{0.5mm}}))\oplus\mathrm{Ind}_{I}^{K}(\theta\otimes\theta^{q}),
	    \\&\mathrm{Sp}(\tau)^{I(1)}=(\theta\otimes\theta)\oplus(\theta^{q}\otimes\theta^{q})\oplus(\theta\otimes\theta^{q})\oplus(\theta^{q}\otimes\theta).
	\end{align*} Let $e_{1}$ be a non-zero vector in the underlying space of the character $\theta\circ\mathrm{det}(\overline{\hspace{0.5mm}\cdot\hspace{0.5mm}})$. Then $e_{2}:=t^{2}e_{1}$ and $e_{0}:=te_{1}$ are the non-zero vectors in the underlying spaces of the characters $\theta^{q}\circ\det$ and $\theta\otimes\theta^{q}$ of $K$ and $I$ respectively. Note that $(\theta\circ\mathrm{det}(\overline{\hspace{0.5mm}\cdot\hspace{0.5mm}}))|_{I}=\theta\otimes\theta$ and
	$(\theta^{q}\circ\mathrm{det}(\overline{\hspace{0.5mm}\cdot\hspace{0.5mm}}))|_{I}=\theta^{q}\otimes\theta^{q}$. The $K$-representation $\mathrm{Ind}_{I}^{K}Ee_{0}$ is stable under the action of $t^{2}$. Thus $t^{2}e_{0}=\varepsilon\varpi^{\nu}f_{0}$ for some $\varepsilon\in \mathcal{O}^{\times}$ and $\nu\in\mathbb{Z}$ where $f_{0}=\sum_{x}u_{x}se_{0}\in(\mathrm{Ind}_{I}^{K}Ee_{0})^{I(1)}$ is a function supported on $IsI(1)$. Let $q=\varepsilon'\varpi^{\nu'}$ where $\varepsilon'\in \mathcal{O}^{\times}$ and $\nu'\in\mathbb{Z}$. The evaluation of the $I(1)$-invariant function $\sum_{x}u_{x}sf_{0}$ on $1$ is $q^{2}$ and on $s$ is $0$. Thus $\sum_{x}u_{x}sf_{0}=q^{2}e_{0}$. Using that $t^{2}\circ(\sum_{x}u_{x}s)=(\sum_{x}u_{x}s)\circ t^{2}$ and that the action of $t^{4}$ is by multiplication by a unit, one obtains that $\nu=-\nu'$. 
	
	Consider the $K$-lattice \[M_{0}=\mathrm{Sym}^{1}\mathcal{O}^{4}\otimes\mathrm{det}^{-\frac{1}{4}}=\mathcal{O}X_{1}\oplus\mathcal{O}X_{2}\oplus\mathcal{O}Y_{1}\oplus\mathcal{O}Y_{2}\] in the representation $\mathrm{Sym}^{1}E^{4}\otimes\mathrm{det}^{-\frac{1}{4}}$. Then $M_{1}=(M_{0}+tM_{0})+t^{2}(M_{0}+tM_{0})$ is a $\mathfrak{K}_{1}$-lattice in $\mathrm{Sym}^{1}E^{4}\otimes\mathrm{det}^{-\frac{1}{4}}$. One computes that \begin{align*}
	    M_{1}&=\varpi_{F}^{\frac{-1}{4}}\mathcal{O}X_{1}\oplus\varpi_{F}^{\frac{-3}{4}}\mathcal{O}X_{2}\oplus\varpi_{F}^{\frac{0}{4}}\mathcal{O}Y_{1}\oplus\varpi_{F}^{\frac{-2}{4}}\mathcal{O}Y_{2}\\&=\varpi^{\frac{-e}{4}}\mathcal{O}X_{1}\oplus\varpi^{\frac{-3e}{4}}\mathcal{O}X_{2}\oplus\varpi^{0}\mathcal{O}Y_{1}\oplus\varpi^{\frac{-2e}{4}}\mathcal{O}Y_{2}\hspace{3mm}
	\end{align*} where $e$ is the ramification index $e(E/F)$.
	Note that $4$ divides $e$ because $E$ is taken large enough to contain $\varpi_{F}^{\frac{1}{4}}$. Hence $\nu'\geq e\geq 4$ and $\nu\leq -4$.
	
	Consider the following $\mathfrak{K}_{1}$-lattice  of $V_{1}$: \begin{align*}
	    L_{1}^{(0)}=L_{1}&:=(\varpi^{0}\mathcal{O}e_{1}\oplus\varpi^{0}\mathcal{O}e_{2}\oplus\varpi^{0}\mathcal{O}e_{0}\oplus\varpi^{\nu}\mathcal{O}f_{0})\otimes M_{1}.
	\end{align*}
For the ease of computation, we write $L_{1}^{(0)}$ as follows: \begin{align*}
		L_{1}^{(0)}&=(0e_{1}\oplus 0e_{2})\otimes(\tfrac{-e}{4}X_{1}\oplus\tfrac{-3e}{4}X_{2}\oplus0Y_{1}\oplus\tfrac{-2e}{4}Y_{2})\\&\oplus (0e_{0}\oplus\nu f_{0})\otimes(\tfrac{-e}{4}X_{1}\oplus\tfrac{-3e}{4}X_{2}\oplus0Y_{1}\oplus\tfrac{-2e}{4}Y_{2}).
	\end{align*} Let us record the actions of $u_{x}s$ and $t$:
\begin{align*}
&u_{x}s((ae_{1}\oplus be_{2})\otimes(n_{1}X_{1}\oplus n_{2}X_{2}\oplus m_{1}Y_{1}\oplus m_{2}Y_{2}))\\&=(ae_{1}\oplus be_{2})\otimes(n_{1}([x]X_{1}+Y_{1})\oplus n_{2}([x^{q}]X_{2}+Y_{2})\oplus m_{1}X_{1}\oplus m_{2}X_{2}),\\&t((ae_{1}\oplus be_{2})\otimes(n_{1}X_{1}\oplus n_{2}X_{2}\oplus m_{1}Y_{1}\oplus m_{2}Y_{2}))\\&=(ae_{0}\oplus(b+\nu) f_{0})\otimes((n_{1}+\tfrac{-e}{4})Y_{2}\oplus (n_{2}+\tfrac{3e}{4})Y_{1}\oplus (m_{1}+\tfrac{-e}{4})X_{1}\oplus (m_{2}+\tfrac{-e}{4})X_{2}),\\&t((ae_{0}\oplus bf_{0})\otimes(n_{1}X_{1}\oplus n_{2}X_{2}\oplus m_{1}Y_{1}\oplus m_{2}Y_{2}))\\&=((b-\nu) e_{1}\oplus ae_{2})\otimes((n_{1}+\tfrac{-e}{4})Y_{2}\oplus (n_{2}+\tfrac{3e}{4})Y_{1}\oplus (m_{1}+\tfrac{-e}{4})X_{1}\oplus (m_{2}+\tfrac{-e}{4})X_{2}).
\end{align*}
		We thus have
	\begin{align*}
	&\sum_{x}u_{x}s((ae_{0}\oplus bf_{0})\otimes(n_{1}X_{1}\oplus n_{2}X_{2}\oplus m_{1}Y_{1}\oplus m_{2}Y_{2}))\\&=  ((b-2\nu)e_{0}\oplus af_{0})\otimes((m_{1}-2\nu)X_{1}\oplus (m_{2}-2\nu)X_{2})\oplus(n_{1}-2\nu)Y_{1}\oplus (n_{2}-2\nu)Y_{2}.
	\end{align*}
	
	We now compute $L_{1}^{(1)}$.
	
	\noindent {\bf Step 1}: 
	We first compute that
	\begin{align*}
	    (\mathfrak{K}_{0}\cdot L_{1}^{(0)})\cap V_{1}&=(0e_{1}\oplus 0e_{2})\otimes(\tfrac{-e}{4}X_{1}\oplus\tfrac{-3e}{4}X_{2}\oplus\tfrac{-e}{4}Y_{1}\oplus\tfrac{-3e}{4}Y_{2})\\&\oplus(0e_{0}\oplus \nu f_{0})\otimes(\tfrac{-e}{4}X_{1}\oplus\tfrac{-3e}{4}X_{2}\oplus0Y_{1}\oplus\tfrac{-2e}{4}Y_{2}).
	\end{align*}
	In this computation, we used $s(e_{i}\otimes X_j)=e_{i}\otimes Y_{j}$ for the first half of the lattice $(\mathfrak{K}_{0}\cdot L_{1}^{(0)})\cap V_{1}$. That the second half of the lattice $(\mathfrak{K}_{0}\cdot L_{1}^{(0)})\cap V_{1}$ is the same as that of $L_{1}^{(0)}$ follows because $\nu\leq -4$ and thus the contribution from the action of $\sum_{x}u_{x}s$ is already in the lattice $L_{1}^{(0)}$. 
	
	\noindent {\bf Step 2}: 
	The action of $t$ on $ (\mathfrak{K}_{0}\cdot L_{1}^{(0)})\cap V_{1}$ gives 
	\begin{align*}
	    t((\mathfrak{K}_{0}\cdot L_{1}^{(0)})\cap V_{1})&=(0e_{0}\oplus \nu f_{0})\otimes(\tfrac{-2e}{4}Y_{2}\oplus0Y_{1}\oplus\tfrac{-2e}{4}X_{1}\oplus\tfrac{-4e}{4}X_{2})\\&\oplus(0e_{1}\oplus 0e_{2})\otimes(\tfrac{-2e}{4}Y_{2}\oplus0Y_{1}\oplus\tfrac{-e}{4}X_{1}\oplus\tfrac{-3e}{4}X_{2}). 
	\end{align*}
Therefore,
	 \begin{align*}
	  L_{1}^{(1)}&=(\mathfrak{K}_{0}\cdot L_{1}^{(0)})\cap V_{1}+t((\mathfrak{K}_{0}\cdot L_{1}^{(0)})\cap V_{1})\\&=(0e_{1}\oplus 0e_{2})\otimes(\tfrac{-e}{4}X_{1}\oplus\tfrac{-3e}{4}X_{2}\oplus\tfrac{-e}{4}Y_{1}\oplus\tfrac{-3e}{4}Y_{2})\\&\oplus(0e_{0}\oplus \nu f_{0})\otimes(\tfrac{-2e}{4}X_{1}\oplus\tfrac{-4e}{4}X_{2}\oplus0Y_{1}\oplus\tfrac{-2e}{4}Y_{2}).
	\end{align*}
	\vspace{1mm}
	
	It follows that 
	\begin{align*}
	  (\mathfrak{K}_{0}\cdot L_{1}^{(1)})\cap V_{1}&=(0e_{1}\oplus 0e_{2})\otimes(\tfrac{-e}{4}X_{1}\oplus\tfrac{-3e}{4}X_{2}\oplus\tfrac{-e}{4}Y_{1}\oplus\tfrac{-3e}{4}Y_{2})\\&\oplus(0e_{0}\oplus \nu f_{0})\otimes(\tfrac{-2e}{4}X_{1}\oplus\tfrac{-4e}{4}X_{2}\oplus0Y_{1}\oplus\tfrac{-2e}{4}Y_{2})=L_{1}^{(1)}. 
	  \end{align*}
	Hence $L_{1}^{(2)}=L_{1}^{(1)}$ and $\pi$ is integral by Corollary \ref{zigzag_criterion}.
	\end{proof}

\bibliography{integral_structures}
\bibliographystyle{amsalpha}

\end{document}